\documentclass[11pt,reqno]{amsart}
\allowdisplaybreaks[4]
\usepackage{amssymb}
\usepackage{anysize}
\usepackage{hyperref}
\usepackage{extarrows}

\numberwithin{equation}{section}
\newtheorem{thm}{Theorem}[section]
\newtheorem{lem}[thm]{Lemma}

\newtheorem{defi}[thm]{Definition}
\newtheorem{rem}[thm]{Remark}
\newtheorem{con}[thm]{Condition}
\newtheorem{Exa}{Example}[section]
\newcommand{\be}{\begin{equation}}
\newcommand{\ee}{\end{equation}}
\newcommand{\bea}{\begin{eqnarray*}}
\newcommand{\eea}{\end{eqnarray*}}

\makeatletter

\newcommand{\Rmnum}[1]{\expandafter\@slowromancap\romannumeral #1@}
\makeatother
\makeatletter 
\@addtoreset{equation}{section}
\makeatother  

\marginsize{35mm}{35mm}{38mm}{40mm}

\begin{document}

\title[]{Universality for a global property of the eigenvectors of Wigner matrices}
\author{Zhigang Bao}
\author{Guangming Pan}
\author{Wang Zhou}
\
\thanks{ Z.G. Bao was partially supported by NSFC grant
11071213, ZJNSF grant R6090034 and SRFDP grant 20100101110001;
G.M. Pan was partially supported by the Ministry of Education, Singapore, under grant \# ARC 14/11;
   W. Zhou was partially supported by the Ministry of Education, Singapore, under grant \# ARC 14/11,  and by a grant
R-155-000-116-112 at the National University of Singapore.}

\address{Department of Mathematics, Zhejiang University, P. R. China}
\email{zhigangbao@zju.edu.cn}

\address{Division of Mathematical Sciences, School of Physical and Mathematical Sciences, Nanyang Technological University, Singapore 637371}
\email{gmpan@ntu.edu.sg}

\address{Department of Statistics and Applied Probability, National University of
 Singapore, Singapore 117546}
\email{stazw@nus.edu.sg}

\subjclass[2010]{15B52, 28C10, 51F25, 60F05}

\date{\today}

\keywords{Wigner matrices, Haar distribution, orthogonal and unitary groups, eigenvectors, Brownian bridge}

\maketitle

\begin{abstract} Let $M_n$ be an $n\times n$ real (resp. complex) Wigner matrix and $U_n\Lambda_n U_n^*$ be its spectral decomposition. Set $(y_1,y_2\cdots,y_n)^T=U_n^*\mathbf{x}$, where $\mathbf{x}=(x_1,x_2,\cdots,$ $x_n)^T$ is a real (resp.  complex) unit vector. Under the assumption that the elements of $M_n$ have $4$ matching moments with those of GOE (resp. GUE), we show that the process $X_n(t)=\sqrt{\frac{\beta n}{2}}\sum_{i=1}^{\lfloor nt\rfloor}(|y_i|^2-\frac1n)$ converges weakly to the Brownian bridge for any $\mathbf{x}$ satisfying $||\mathbf{x}||_\infty\rightarrow 0$ as $n\rightarrow \infty$, where $\beta=1$ for the real case and $\beta=2$ for the complex case. Such a result indicates that the othorgonal (resp. unitary) matrices with columns being the eigenvectors of Wigner matrices are asymptotically Haar distributed on the orthorgonal (resp. unitary) group from a certain perspective.
\end{abstract}
\section{Introduction}
Let $M_n=\frac{1}{\sqrt{n}}(v_{ij})_{n,n}$ be an $n\times n$ real or complex Wigner matrix whose definition is stated below.
\begin{defi}[Real Wigner matrix] We call $M_n$ a real Wigner matrix if it is a symmetric random matrix such that $\{v_{ij}, 1\leq i\leq j\leq n\}$ is a collection of independent real random variables. And
$v_{ij}, 1\leq i<j\leq n$ are i.i.d with common mean $0$ and variance $1$. And $v_{ii}, 1\leq i\leq n$ are i.i.d with common mean $0$ and finite variance.
\end{defi}
\begin{defi}[Complex Wigner matrix] We call $M_n$ a complex Wigner matrix if it is an Hermitian random matrix such that $\{v_{ij}, 1\leq i\leq j\leq n\}$ is a collection of independent random variables. And
$v_{ij}, 1\leq i<j\leq n$  are i.i.d complex random vaiables with common mean $0$ and variance $1$. And $v_{ii}, 1\leq i\leq n$ are i.i.d real random variables with common mean $0$ and finite variance.
\end{defi}
\begin{Exa}[GOE and GUE] A real Wigner matrix with $N(0,1)$ off-diagonal elements and $N(0,2)$ diagonal elements is called GOE. And a complex Wigner matrix with $N(0,1/2)+\sqrt{-1}N(0,1/2)$ above-diagonal elements and $N(0,1)$ diagonal elements is called GUE. Here $N(0,1/2)+\sqrt{-1}N(0,1/2)$ stands for the standard complex normal  variable whose real and imaginary parts are independent.
\end{Exa}
Now we denote the ordered eigenvalues of $M_n$ by $\lambda_1\leq\cdots\leq \lambda_n$ and the corresponding normalized eigenvectors by $\mathbf{u}_1,\mathbf{u}_2,\cdots,\mathbf{u}_n$. Set $\Lambda_n={\rm{diag}}(\lambda_1,\cdots,\lambda_n)$ and write the spectral decomposition of $M_n$ as
\begin{eqnarray*}
M_n=U_n\Lambda_n U_n^*,
\end{eqnarray*}
where
\begin{eqnarray*}
U_n=(\mathbf{u}_1, \mathbf{u}_2,\cdots, \mathbf{u}_n).
\end{eqnarray*}
Conventionally, we require the coefficients of the eigenvectors to be real in the real case. However, the choices of the normalized eigenvectors are not unique owing to the following two reasons.\\

({\bf{r1}}): {\emph{If there is an $i\in\{1,2,\cdots,n\}$ such that $\lambda_i$ is not simple, one can arbitrarily choose an orthogonal basis of the eigenspace corresponding to $\lambda_i$.}}\\

({\bf{r2}}): {\emph{If every eigenvalue is simple, we can rotate the eigenvector by multiplying a sign $-1$ in the real case or any phase $e^{\sqrt{-1}\theta} (\theta\in\mathbb{R})$ in the complex case.}}\\

Note that, when the matrix elements are continuously distributed, ({\bf{r1}}) will cause no ambiguity since the eigenvalues are simple with probability one in this case. Moreover, even for the discontinuous case, it will be shown that whichever eigenvectors for the multiple eigenvalues have been taken will not affect the results of this paper (see the discussion at the end of Section 2 ).  And actually, ({\bf{r2}}) will also cause no trouble since the concerned quantities in this paper only depend on the projections
\[\mathbf{u}_i\mathbf{u}_i^*,\quad i=1,\cdots,n\]
 which are uniquely defined as long as the eigenvalues are simple. However, in order to eliminate the ambiguities in some issues we will adopt the following viewpoint to fix the definition of the eigenvector $\mathbf{u}_i$ when $\lambda_i$ is simple.\\

 {\bf{A viewpoint}}: \emph{In the complex case, one can replace $\mathbf{u}_i$ by $e^{\sqrt{-1}\theta_i}\mathbf{u}_i$, where $\theta_i$ is uniformly distributed on $[0,2\pi)$. Moreover, $\theta_i, i=1,\cdots, n$ are i.i.d and independent of the matrix $M_n$. In the real case, one can replace $\mathbf{u}_i$ by $b_i\mathbf{u}_i$, where $b_i,i=1,\cdots,n$ are i.i.d $\pm 1$ Bernoulli variables which are independent of $M_n$.}\\

Under the above viewpoint,  it is well known that when $M_n$ is GOE (resp. GUE), $U_n$ is Haar distributed on the orthogonal group $O(n)$ (resp. unitary group $U(n)$). Then it is natural to conjecture that $U_n$ of the general Wigner matrices is ``asymptotically'' Haar distributed in some sense. In other words, we care about the universal properties of the matrices of eigenvectors. In the past decades, a vast of work had been devoted to the study of the universality problems of various statistics of the eigenvalues. By contrast, the work on the universality of eigenvectors is much less. The most recent progresses on this aspect maybe the delocalization or localization property of the eigenvectors for different types of random matrices (see \cite{ESY2009}, \cite{ESY20092}, \cite{BP20121}, \cite{BG20121}, \cite{EK2011} and \cite{S2009} for instance) and the universality for the local statistics of the eigenvector coefficients (see \cite{KY2011}, \cite{TV2012}).

In this paper, we will discuss a universality result for a global property of the eigenvectors. Below we give the definition of the concerned quantity of our paper and then explain why it is closely related to the universality of the distribution of $U_n$.

Let $\mathbf{x}=(x_1,\cdots,x_n)^T$ be a  definite unit vector. That is to say, $\mathbf{x}\in S^{n-1}$ in the real case, and $\mathbf{x} \in S^{2n-1}$ in the complex case. Here
\begin{eqnarray*}
S^{n-1}:=\{\mathbf{r}\in\mathbb{R}^n:||\mathbf{r}||=1\},\quad
S^{2n-1}:=\{\mathbf{z}\in\mathbb{C}^n: ||\mathbf{z}||=1\}.
\end{eqnarray*}
Now we set the vector
\begin{eqnarray*}
\mathbf{y}=(y_1,\cdots,y_n)^T=U_n^*\mathbf{x}.
\end{eqnarray*}
Then we can construct a process $X_n(t) \in D[0,1]$
from the vector $\mathbf{y}$ as
\begin{eqnarray}
X_n(t)=\sqrt{\frac{\beta n}{2}} \sum_{i=1}^{\lfloor nt\rfloor}(|y_i|^2-\frac1n), \label{p.1}
\end{eqnarray}
where $\beta=1$ is for the real case and $\beta=2$ is for the complex case. Hereafter, the notation $\lfloor x\rfloor$ stands for integer part of $x$. In this paper, we will discuss the limit of the process (\ref{p.1}) in the weak sense. Such a problem was raised by Silverstein in \cite{Silverstein1990} and was shown to be closely related to the universality problem on $U_n$.

Note that when $U_n$ is Haar distributed on the orthogonal group $O(n)$ (resp. unitary group $U(n)$), it is well known that for any real (resp. complex) unit vector $\mathbf{x}$ one has that $\mathbf{y}$ is uniformly distributed on $S^{n-1}$ (resp. $S^{2n-1}$). Then for the real case, one has
\begin{eqnarray*}
\mathbf{y}\xlongequal{d} \frac{\mathbf{g}_{\mathbb{R}}}{||\mathbf{g}_\mathbb{R}||},
\end{eqnarray*}
where $\mathbf{g}_\mathbb{R}:=(g_1,g_2,\cdots,g_n)^T$ is a Gaussian vector with i.i.d. $N(0,1)$ coefficients. Similarly, for the complex case one has
\begin{eqnarray*}
\mathbf{y}\xlongequal{d} \frac{\mathbf{g}_{\mathbb{C}}}{||\mathbf{g}_\mathbb{C}||},
\end{eqnarray*}
where $\mathbf{g}_{\mathbb{C}}:=(\eta_1+
\sqrt{-1}\zeta_1,\cdots,\eta_n+\sqrt{-1}\zeta_n)^T$ is a Gaussian vector with i.i.d. $N(0,1/2)$ $+\sqrt{-1}N(0,1/2)$ coefficients. Note that actually we can choose the Gaussian variables with arbitrary common variance because of the scaling invariance. Here we just specify them to be standard for convenience. Then by using the classical results on weak convergence, it is elementary to see for the Gaussian case
\begin{eqnarray}
X_n(t)\stackrel{d}\Longrightarrow W^\circ(t),\quad t\in[0,1], \label{0.1}
\end{eqnarray}
where $W^\circ(t)$ is the standard Brownian bridge.

Conversely, (\ref{0.1}) reflects the fact that $\mathbf{y}$ is uniformly distributed on sphere in a
global sense. Thus if (\ref{0.1}) is valid for general Wigner matrices with a large class of $\mathbf{x}$, we can regard that $U_n$ of general case is ``asymptotically'' Haar distributed from such a certain perspective. Such a ``measure'' of closeness between the distribution of $U_n$ and the Haar distribution was first raised by Silverstein in \cite{Silverstein1990} for the sample covariance matrices. However, Silverstein only succeeded in proving the result for the unit vectors $\mathbf{x}=(\pm1/\sqrt{n},\cdots,\pm1/\sqrt{n})^T$ under the assumption that the matrix elements are symmetrically distributed. As discussed above, in the Gaussian case, $\mathbf{x}$ can be arbitrary. Thus it is crucial to verify (\ref{0.1}) for more general $\mathbf{x}$ rather than those in \cite{Silverstein1990}. In this paper, we will prove (\ref{0.1}) for a large class of Wigner matrices under the restriction of $||\mathbf{x}||_\infty\rightarrow 0$  which will be shown to be necessary for the concerned problem in the general distribution cases (see Remark \ref{rem.10.9} below). Here $||\mathbf{x}||_\infty=\max_{1\leq i\leq n}|x_i|$ is the maximum norm of $\mathbf{x}$. Moreover, we do not need the symmetrical distribution  condition imposed on the matrix elements.

To state our main result, we need an ad hoc terminology.
\begin{defi}[Matching to the $k$-th moments] We say that two Wigner matrices $M_n=1/\sqrt{n}(v_{ij})$ and $\widetilde{M}_n=1/\sqrt{n} (\tilde{v}_{ij})$ match to the $k$-th moments if for all $1\leq i,j\leq n$,
\begin{eqnarray*}
\mathbb{E}(\Re(v_{ij})^l\Im(v_{ij})^m)=\mathbb{E}(\Re(\tilde{v}_{ij})^l\Im(\tilde{v}_{ij})^m),\quad 0\leq l,m\leq l+m\leq k.
\end{eqnarray*}
\end{defi}
In the sequel, we will specify $k=4$. That means we require the elements of two concerned Wigner matrices have the same first four moments. Moreover, throughout the paper, we will need the following additional condition on the matrix elements.
\begin{con}\label{con.1} We assume the matrix elements $v_{ij}$'s have uniform subexponential dacay. That is,
\begin{eqnarray*}
\mathbb{P}(|v_{ij}|\geq t) \leq C^{-1}\exp(-t^C)
\end{eqnarray*}
with some positive constant $C$ independent of $i,j$.
\end{con}

Now we can state our main result.
\begin{thm}  \label{th.1}Assume that $M_n$ is a real (resp. complex) Wigner matrix matching GOE (resp. GUE) to the $4$-th moments. Moreover, we assume that $M_n$ satisfies Condition \ref{con.1}. For any definite real (resp. complex) unit vector $\mathbf{x}$ satisfying $||\mathbf{x}||_\infty\rightarrow 0$ as $n$ tends to infinity, we have
\begin{eqnarray*}
X_n(t)=\sqrt{\frac{\beta n}{2}}\sum_{i=1}^{\lfloor nt\rfloor}(|y_i|^2-\frac1n)\Longrightarrow W^\circ(t).
\end{eqnarray*}
Here $\beta=1,2$ in the real case and complex case respectively.
\end{thm}
Hereafter, when we refer to ``limit'' and ``accumulation point'' of a random sequence, they are always in the sense of weak convergence. Moreover, for simplicity, when there is no confusion, we may omit the time parameter $t$ from $X_n(t)$ and $W^\circ(t)$.

The main proof strategy will benefit from the discussions in \cite{Silverstein1990}. Specifically, a criteria for weak convergence for a random sequence on $D[0,1]$ with its limit supported on $C[0,1]$ was provided in \cite{Silverstein1990} (see Theorem 3.1 of \cite{Silverstein1990}). Such a criteria can be regarded as a slight modification of the classical ``$finite~dimensional~convergence+tightness$'' issue.  The discussions in \cite{Silverstein1990} and the recent result of Bai and Pan \cite{BP2012} can help us to confirm that the unique possible $C[0,1]$-supported accumulation point of $(X_n)_{n\geq 1}$ is $W^\circ$.  Then it remains to show that $(X_n)_{n\geq 1}$ is tight and can only has $C[0,1]$-supported accumulation point.  However, it has been shown in \cite{Silverstein1990} that the proof of the tightness of the sequence $(X_n)_{n\geq 1}$ is an obstacle to this problem. In order to show the tightness, Silverstein imposed the additional symmetrical distribution condition on the matrix elements and restricted the discussions on the special cases of $\mathbf{x}=(\pm1/\sqrt{n},\cdots,\pm1\sqrt{n})^T$ in \cite{Silverstein1990} for the sample covariance matrices. Actually, Silverstein's proof can be adopted after slight modifications to the Wigner matrices under similar restrictions as those imposed in \cite{Silverstein1990} for the sample covariance matrices.

To remove the restrictions mentioned above, we will use a totally different method. A new input is the so-called isotropic local semicircle law proposed by Knowles and Yin in \cite{KY2012} quite recently. Crudely speaking, we can verify the tightness of $(X_n)_{n\geq 1}$ through providing some  good upper bounds on the fourth moments of the increments of $X_n(t)$. Such bounds will turn out to be easily obtained for the Gaussian case owing to the explicit distribution information of $\mathbf{y}$. For more general Wigner matrices, we will use the idea of comparing the general case with the Gaussian case. Such a comparison method relies on the classical Linderberg strategy, i.e. replacing the matrix elements by those of the ``reference'' matrix one pair (or one unit in the diagonal case) each time and then evaluating the change of the concerned quantity induced by the replacement on each step. Then by a telescoping argument we can get the difference of the concerned quantities of two Wigner matrices. Such an approach was used in the literature of the Random Matrix Theory recently. One can see \cite{TV2011}, \cite{EYY2010} and \cite{EYY2012} for instances. Particularly, one can refer to \cite{KY2011}, \cite{TV2012} and \cite{BG2012}  for the applications of such a strategy on some problems about the eigenvectors of the Wigner matrices.

More precisely, to provide the upper bounds on the fourth moments of the increments of $X_n(t)$, we will mainly pursue the idea of the Green function comparison approach raised by Erd\"{o}s, Yau and Yin in \cite{EYY2010}. To this end, at first, we will approximate the increment of the process by a quantity expressed in terms of the Green function. Then we will perform a replacement issue on the Green functions to achieve the purpose of comparison. Such a strategy will rely on the isotropic local semicircle law provided in \cite{KY2012}.

Our paper will be organized as follows. In Section 2,  we will present some necessary preliminaries. And in section 3, we will provide a criteria of the weak convergence of $X_n(t)$ which contains two statements. It will be shown that the first statement can be implied by a recent result of Bai and Pan \cite{BP2012}, thus we will just sketch the proof of this statement at the end of Section 3. The second statement is mainly about the tightness of the sequence $(X_n)_{n\geq 1}$, which will be handled in Section 4.

Throughout the paper, the notations $C, C_1, C'$ and $K$ will be used to denote some $n$-independent positive constants whose values may defer from line to line. The notation $||\cdot||_{op}$ stands for the operator norm of a matrix.

We will say an event $\mathbf{E}$ occurs with overwhelming probability if and only if
\begin{eqnarray*}
\mathbb{P}(\mathbf{E})\geq 1-n^{-K}
\end{eqnarray*}
for any given positive number $K$ when $n$ is sufficiently large.
\section{Preliminaries}
In this section, we will state some basic notions and recent results, especially some known results on the Green functions which will be frequently used in the proof of Theorem \ref{th.1}.

The so-called empirical spectral distribution (ESD) of $M_n$ is defined by
\begin{eqnarray*}
F_n(x):= \frac{1}{n} \sum_{i=1}^{n}\mathbf{1}_{\{\lambda_i\leq x\}}.
\end{eqnarray*}
It is well known that $F_n(x)$ almost surely converges weakly to Wigner's semicircle law $F_{sc}(x)$ whose density function is given by
\begin{eqnarray}
\rho_{sc}(x)=\frac{1}{4\pi^2} \sqrt{4-x^2}\mathbf{1}_{\{|x|\leq 2\}}. \label{p.0}
\end{eqnarray}
The Stieltjes transform of a probability measure $\mu$ can be defined for all complex number $z=E+\sqrt{-1}\eta\in \mathbb{C}\setminus \mathbb{R}$ as
\begin{eqnarray*}
m^{\mu}(z)=\int\frac{1}{x-z} \mu(dx).
\end{eqnarray*}
Here $E$ and $\eta$ are the real and imaginary parts of $z$ respectively. Thus by definitions, we have
\begin{eqnarray}
m^{F_n}(z)=\frac{1}{n}\sum_{i=1}^n\frac{1}{\lambda_i-z}=\frac{1}{n}tr(M_n-zI)^{-1}, \label{p.3}
\end{eqnarray}
and
\begin{eqnarray*}
m^{F_{sc}}(z)=\int_{-2}^2\frac{1}{x-z}\rho_{sc}(x)dx.
\end{eqnarray*}
For simplicity of the notation, we will briefly write $m^{F_n}(z)$ and $m^{F_{sc}}(z)$ as $m_n(z)$ and $m_{sc}(z)$ respectively. It is well known that
\begin{eqnarray*}
\sup_{z\in\mathbb{C}\setminus \mathbb{R}}|m_{sc}(z)|=O(1).
\end{eqnarray*}

The Green function $G_n(z)$ of $M_n$ is defined by
\begin{eqnarray*}
G_n(z):=(G_{ij}(z))_{n,n}=(M_n-zI_n)^{-1}
\end{eqnarray*}
which is also called the resolvent of $M_n$. Now by (\ref{p.3}), we also have
\begin{eqnarray*}
m_n(z)=\frac{1}{n}trG_n(z)=\frac{1}{n}\sum_{i}G_{ii}(z).
\end{eqnarray*}
It was shown by Erd\"{o}s, Yau and Yin in \cite{EYY2012} that when
\begin{eqnarray}
z\in \mathbf{S}:=\{E+\sqrt{-1}\eta: |E|\leq 5, n^{-1}(\log n)^{C\log\log n}< \eta\leq 10\} \label{101010}
\end{eqnarray}
for some positive constant $C$, one has that $m_n(z)$ is well approximated by $m_{sc}(z)$ with overwhelming probability. Moreover, it was proved in \cite{EYY2012} that $G_{ii}(z)$'s are close to $m_{sc}(z)$ and $G_{jk}(z)$'s $(j\neq k)$ are small in the sense that for some positive constant $C$,
\begin{eqnarray}
\max_i |G_{ii}(z)-m_{sc}(z)|+\max_{j\neq k}|G_{jk}(z)|\leq (\log n)^{C\log\log n} \left(\sqrt{\frac{\Im m_{sc}(z)}{n\eta}}+\frac{1}{n\eta}\right) \label{0.0.0}
\end{eqnarray}
holds uniformly for $z\in\mathbf{S}$
with overwhelming probability. See Theorem 2.1 of \cite{EYY2012} for details.

Now if we denote the standard basis of $\mathbb{R}^n$ by $\mathbf{e}_1,\mathbf{e}_2,\cdots,\mathbf{e}_n$ conventionally, i.e. $\mathbf{e}_i$ is the $n\times 1$ vector with only the $i$-th component being $1$ and the others being $0$. Then we can write
\begin{eqnarray*}
G_{ij}(z)=\mathbf{e}_i^*G_n(z)\mathbf{e}_j.
\end{eqnarray*}
Recently, Knowles and Yin generalized the estimation (\ref{0.0.0}) to the quantities
\begin{eqnarray*}
G_{\mathbf{v}\mathbf{w}}:=
\mathbf{v}^*
G_n(z)\mathbf{w}
\end{eqnarray*}
for any definite unit vectors $\mathbf{v}, \mathbf{w}$ in \cite{KY2012} and provided the so-called isotropic local semicircle law. Meanwhile, for any unit vector $\mathbf{v}$ they also provided in \cite{KY2012} the uniform upper bounds for the quantities
\begin{eqnarray*}
|\langle \mathbf{u}_i,\mathbf{v}\rangle|, i=1,\cdots, n.
\end{eqnarray*}
And they named the control on the quantities above as the isotropic delocalization of the eigenvectors, which can be viewed as a generalization of the delocalization property for eigenvectors raised in \cite{ESY2009}. Both the isotropic local semicircle law and isotropic delocalization property will be crucial to our analysis in the sequel. We remark here the assumptions imposed in \cite{KY2012} are weaker than those made in our paper. We refer to \cite{KY2012} for details and will not mention this fact again in the sequel. For convenience, we will reformulate their results as the following lemma under our assumptions.
\begin{lem}[Knowles and Yin, \cite{KY2012}] Under the assumptions of Theorem \ref{th.1}, we have the following two statements.\\

(1):(isotropic local semicircle law). For $z\in\mathbf{S}$, there exists some positive constant $C$ such that
\begin{eqnarray}
|\langle \mathbf{v}, G_n(z)\mathbf{w}\rangle-m_{sc}(z)\langle \mathbf{v}, \mathbf{w}\rangle |\leq (\log n)^{C\log\log n}\left(\sqrt{\frac{\Im m_{sc}(z)}{n\eta}}+\frac{1}{n\eta}\right) \label{p.5}
\end{eqnarray}
with overwhelming probability for all deterministic and normalized vectors $\mathbf{v},\mathbf{w}\in \mathbb{C}^n$.\\

(2):(isotropic delocalization). For any deterministic and normalized vector $\mathbf{v}\in \mathbb{C}^n$, we have
\begin{eqnarray}
\sup_{i}|\langle \mathbf{u}_i,\mathbf{v}\rangle|^2\leq \frac{(\log n)^{C\log\log n}}{n} \label{p.6}
\end{eqnarray}
for some positive constant $C$ with overwhelming probability.
\end{lem}
\begin{rem} We remind here that the validity of (\ref{p.6}) does not depend on the ambiguity of the choices of the eigenvectors caused by ({\bf{r1}}) and ({\bf{r2}}). For details, see the proof of Theorem 2.5 of \cite{KY2012}.
\end{rem}
\begin{rem} Actually, we will need (\ref{p.5}) to hold uniformly for all $z\in \mathbf{S}$ with overwhelming probability in some discussions below. Note that
\begin{eqnarray}
||G'(z)||_{op}\leq \eta^{-2}. \label{p.7}
\end{eqnarray}
Now we choose an $\varepsilon$-net of $\mathbf{S}$ with $\varepsilon=n^{-K}$ with sufficiently large $K$. Then we have (\ref{p.5}) holds uniformly on the $\varepsilon$-net with overwhelming probability. By using (\ref{p.7}) and the elementary mean value theorem, we can get that (\ref{p.5}) uniformly holds on $\mathbf{S}$ with overwhelming probability by slightly adjusting the constant $C$ in (\ref{p.5}).
\end{rem}

At the end of this section we explain that the ambiguity caused by ({\bf{r1}}) does not influence the limit property of $X_n(t)$. Now let $\gamma_i:=\gamma_{i,n}\in[-2,2]$ be the classical location of $\lambda_i$ in the sense that
\begin{eqnarray*}
\int_{-2}^{\gamma_{i}}\rho_{sc}(x)dx=\frac{i}{n}.
\end{eqnarray*}
It is easy to check that for some positive constant $C$,
\begin{eqnarray}
|\gamma_{i}-\gamma_{i+1}|\leq C [\min (i,n-i+1)]^{-1/3}n^{-2/3},  \quad i=1,\cdots,n-1. \label{6.6.6.6.6}
\end{eqnarray}
By the rigidity property of the eigenvalues which was proved by Erd\"{o}s, Yau and Yin in \cite{EYY2012} we see that with overwhelming probability, the event
\begin{eqnarray}
\bigcap_{i=1}^n\{|\lambda_i-\gamma_i|\leq (\log n)^{C_1\log\log n}[\min (i,n-i+1)]^{-1/3}n^{-2/3}\} \label{7.7.7.7.7}
\end{eqnarray}
holds for some positive constant $C_1$ when $n$ is sufficiently large.
Now we assume that there is an $n_0$ such that
\begin{eqnarray*}
\lambda_{n_0}<\lambda_{n_0+1}
=\cdots=\lambda_{n_0+k}
<\lambda_{n_0+k+1}.
\end{eqnarray*}
 Note that although the eigenvectors  $\mathbf{u}_{n_0+1},\cdots, \mathbf{u}_{n_0+k}$ can be chosen in many different ways, the choice of the projection matrix
 \begin{eqnarray*}
 (\mathbf{u}_{n_0+1},\cdots, \mathbf{u}_{n_0+k})(\mathbf{u}_{n_0+1},\cdots, \mathbf{u}_{n_0+k})^*
 \end{eqnarray*}
 is unique. Thus the quantity
 \begin{eqnarray*}
 \sum_{i=n_0+1}^{n_0+k}(|y_i|^2-\frac1n)=\mathbf{x}^*
 (\mathbf{u}_{n_0+1},\cdots, \mathbf{u}_{n_0+k})(\mathbf{u}_{n_0+1},\cdots, \mathbf{u}_{n_0+k})^*
 \mathbf{x}-\frac{k}{n}
 \end{eqnarray*}
 is uniquely defined. This shows that the definition of $X_n(t)$ does not depend on the choices of the eigenvectors as long as $\lambda_{\lfloor nt\rfloor}$ is a simple eigenvalue. Now if $\lambda_{\lfloor nt\rfloor}$ is not simple, we can assume that $n_0+1\leq \lfloor nt\rfloor\leq n_0+k$ without loss of generality. Following from (\ref{6.6.6.6.6}) and (\ref{7.7.7.7.7}), it is not difficult to see with overwhelming probability, there is no eigenvalue with multiplicity larger than $(\log n)^{2C_1\log \log n}$. For $n_0+1\leq \lfloor nt\rfloor\leq n_0+k$, we can write
 \begin{eqnarray}
 X_n(t)=X_n(n_0/n)+\sqrt{\frac{\beta n}{2}}\sum_{i=n_0+1}^{\lfloor nt\rfloor}(|y_i|^2-\frac{1}{n}).\label{8.8.8.8.8}
 \end{eqnarray}
 Then by the fact $k\leq (\log n)^{2C_1\log \log n}$ with overwhelming probability and the isotropic delocalization property (\ref{p.6}) we see that the second term on the right hand side of (\ref{8.8.8.8.8}) (not well defined term) can be discarded in probability. Moreover, since both the upper bound of the multiplicity of  eigenvalue and isotropic delocalization property hold uniformly in $i=1,\cdots, n$, the above discussion also holds uniformly in $t\in[0,1]$. Thus, the limit behaviour of $X_n(t)$ does not depend on the ambiguity caused by ({\bf{r1}}).

\section{Uniqueness of the $C[0,1]$-supported accumulation point}
In this section, we will provide some known results and mainly show that $W^\circ$ is the unique $C[0,1]$-supported accumulation point of $(X_n)_{n\geq 1}$. Similar to the discussions in \cite{BG2012}, to prove Theorem \ref{th.1}, it suffices to verify the following two lemmas.\\

\begin{lem} \label{lem.1.1.1.1}Under the assumptions of Theorem \ref{th.1},
the sequence $(X_n)_{n\geq 1}$ has $W^\circ$ as its unique possible accumulation point supported on $C[0,1]$.
\end{lem}

\begin{lem}\label{lem.2.2.2.2} Under the assumptions of Theorem \ref{th.1}, the sequence $(X_n)_{n\geq 1}$ is tight and can only have $C[0,1]$-supported accumulation points.
\end{lem}

The remaining part of this section will be devoted to the proof of Lemma \ref{lem.1.1.1.1}.  To this end , we will need the Theorem 1.2 of \cite{BP2012}. For convenience of the reader, we rewrite it here. Let
\begin{eqnarray*}
W_n(g)=\sqrt{\beta n} (\mathbf{x}^*g(M_n)\mathbf{x}-\frac1ntr g(M_n)),
\end{eqnarray*}
where $g(x)$ is a function analytic on a region in the complex plane covering the interval $[-2,2]$. Then we have the following theorem provided by Bai and Pan in \cite{BP2012}
\begin{thm} [Bai and Pan, \cite{BP2012}]\label{thm.2.1}
Let $M_n$ be a real or complex Wigner matrix satisfying $\mathbb{E}|v_{12}|^4<\infty$. Suppose that $g_1,\cdots,g_k$ are analytic on an open interval including $[-2,2]$, and that
\begin{eqnarray*}
||\mathbf{x}||_\infty\rightarrow 0.
\end{eqnarray*}
(1) If $M_n$ is real, i.e. $\beta=1$, and $\mathbb{E}v_{12}^3=0$, then $W_n(g_1),\cdots, W_n(g_k)$ converges weakly to a Gaussian vector $W_f$ with mean zero and covariance function
\begin{eqnarray}
Cov(W_{g_1},W_{g_2})=2\left(\int g_1(x)g_2(x)dF_{sc}(x)-\int g_1(x)dF_{sc}(x)\int g_2(x) dF_{sc}(x)\right).\label{10.1}
\end{eqnarray}
(2) If $M_n$ is complex, i.e. $\beta=2$, and $\mathbb{E}v_{12}^2=0$ and $\mathbb{E}v_{12}^2\bar{v}_{12}=0$, then (1) remains true.
\end{thm}
\begin{rem} We remind here that $W_n(g)$ in the complex case is different from $X_n(g)$ in \cite{BP2012} in scaling.
\end{rem}
\begin{rem} \label{rem.10.9} It has been shown that the condition $||\mathbf{x}||_\infty\rightarrow 0$ is necessary for Theorem \ref{thm.2.1}. See Remark 1.4 of \cite{BP2012}.
\end{rem}
Now we begin to prove Lemma \ref{lem.1.1.1.1}.
\begin{proof}[Proof of Lemma \ref{lem.1.1.1.1}]
Changing the variable $t$ by $F_{sc}(u)$, one can see that it is equivalent to verify that
the sequence $(X_n(F_{sc}(u)))_{n\geq 1}$ has $W^\circ(F_{sc}(u))$ as its unique possible accumulation point supported on $C[-2,2]$.
We claim that it suffices to show the following two statements. \\

({\bf{a}}):\emph{ We have
 \begin{eqnarray*}
 \left\{\int_{-2}^2u^r X_n(F_{sc}(u))du\right\}_{r=0}^\infty\Longrightarrow
 \left\{\int_{-2}^2u^rW^\circ({F_{sc}(u)})du\right\}_{r=0}^\infty.
 \end{eqnarray*}}

 ({\bf{b}}): \emph{The distribution of a process $X(u)$ supported on $C[-2,2]$ is uniquely determined by the distribution of
 \begin{eqnarray*}
 \left\{\int_{-2}^2u^rX(u)du\right\}_{r=0}^\infty.
 \end{eqnarray*}}
Below we sketch the proof of Lemma \ref{lem.1.1.1.1} providing ({\bf{a}}) and ({\bf{b}}) at first.
Note that if we assume that one convergent subsequence $\{X_{n'}(F_{sc}(u))\}$ converges weakly to some $C[-2,2]$-supported process $X(u)$, then by Theorem 5.1 of \cite{Billingsley1968}, one has
\begin{eqnarray*}
 \left\{\int_{-2}^2u^r X_{n'}(F_{sc}(u))du\right\}_{r=0}^\infty\Longrightarrow
 \left\{\int_{-2}^2u^r X(u)du\right\}_{r=0}^\infty.
 \end{eqnarray*}
Meanwhile, by ({\bf{b}}) we also know that if $X(u)$ is  $C[-2,2]$-supported, its distribution is uniquely determined by the distribution of
  \begin{eqnarray*}
 \left\{\int_{-2}^2u^rX(u)du\right\}_{r=0}^\infty.
 \end{eqnarray*}
 Thus we have $\{X_{n'}(F_{sc}(u))\}$ converges weakly to $W^\circ(F_{sc}(u))$ as $n'\rightarrow\infty$. Therefore, we have Lemma \ref{lem.1.1.1.1} by ({\bf{a}}) and ({\bf{b}}).
 It remains to verify ({\bf{a}}) and ({\bf{b}}). The proof of ({\bf{b}}) is nearly the same as the counterpart of the proof for Theorem 3.1 of \cite{Silverstein1990}. Thus here we omit the detail.

To verify ({\bf{a}}) for $X_n(F_{sc}(u))$, we will work on the slight modification $X_n(F_{n}(u))$  instead. Note that by the rigidity property which was proved by Erd\"{o}s, Yau and Yin in \cite{EYY2012} we see that there exists some positive constant $C$  such that
\begin{eqnarray}
\sup_{|u|\leq 5} |F_n(u)-F_{sc}(u)|\leq \frac{(\log n)^{C\log\log n}}{n} \label{8889}
\end{eqnarray}
with overwhelming probability (see Theorem 2.2 of \cite{EYY2012}).
Thus we have for some positive constant $C'$
\begin{eqnarray*}
\sup_{|u|\leq 5}|X_n(F_{sc}(u))-X_n(F_n(u))|\leq \sqrt{n}(\log n)^{C\log\log n} \max_{i}(|y_i|^2+\frac{1}{n})\leq \frac{(\log n)^{C'\log\log n}}{\sqrt{n}}
\end{eqnarray*}
with overwhelming probability. Above we have used the isotropic delocalization property (\ref{p.6}). Therefore, it suffices to study the limit behaviour of
\begin{eqnarray*}
\left\{\int_{|u|\leq 5}u^rX_n(F_n(u))du\right\}
_{r=0}^{\infty}.
\end{eqnarray*}
Moreover, also by the rigidity property provided in \cite{EYY2012}, we see that all the eigenvalues of $M_n$ are in the interval $[-5,5]$ with overwhelming probability. Combining with the fact that
\begin{eqnarray*}
X_n(0)=X_n(1)=0,
\end{eqnarray*}
we have
\begin{eqnarray*}
\left\{\int_{|u|\leq 5}u^rX_n(F_n(u))du\right\}
_{r=0}^{\infty}=\left\{
\int_{-\infty}^{+\infty}u^rX_n(F_n(u))du\right\}
_{r=0}^{\infty}
\end{eqnarray*}
with overwhelming probability.
Relying on the discussion above one can transfer the problem to show that
 \begin{eqnarray*}
 \left\{\int_{-\infty}^\infty u^r X_n(F_n(u))du\right\}_{r=0}^\infty\Longrightarrow
 \left\{\int_{-2}^2u^r W^\circ(F_{sc}(u))du\right\}_{r=0}^\infty.
 \end{eqnarray*}
 By integration by parts, it suffices to verify
 \begin{eqnarray}
 \left\{\int_{-\infty}^\infty u^r dX_n(F_n(u))\right\}_{r=0}^\infty\Longrightarrow
 \left\{\int_{-2}^2u^rd W^\circ(F_{sc}(u))\right\}_{r=0}^\infty. \label{3.3.3.3}
 \end{eqnarray}

 Note that
 \begin{eqnarray*}
 \int_{-\infty}^\infty u^r dX_n(F_n(u))=\sqrt{\frac{\beta n}{2}}\left(\mathbf{x}^*M_n^r\mathbf{x}-\frac1ntr M_n^r\right).
 \end{eqnarray*}
 Thus we arrive at the stage to use Theorem \ref{thm.2.1}. Note that Theorem \ref{thm.2.1} only depends on the first three moments of the matrix elements. And the GOE and GUE obviously satisfy the moment assumptions in Theorem \ref{thm.2.1}. Moreover, by the discussions above and (\ref{0.1}) for the Gaussian case, (\ref{3.3.3.3}) is valid for GOE and GUE obviously. Hence, (\ref{3.3.3.3}) also holds for general Wigner matrices under the assumptions of Theorem \ref{th.1}. So we complete the proof.
 \end{proof}
\section{Tightness of $(X_n)_{n\geq 1}$}
In this section, we will prove Lemma \ref{lem.2.2.2.2}. At first, we show that the process sequence $(X_n)_{n\geq 1}$ can only have $C[0,1]$-supported accumulation points. It suffices to check that the maximal jump of the process $X_n(t)$ converges to zero in probability. This can be seen directly from the isotropic delocalization property (\ref{p.6}). Thus the remaining part of this section will be devoted to showing the tightness of the process sequence $(X_n)_{n\geq 1}$, which is the main part of our proof. To this end, we begin with the modulus of continuity of the process $X_n$ as
\begin{eqnarray*}
w_{X_n}(\delta)=w(X_n,\delta):=\sup_{|t_1-t_2|<\delta}|X_n(t_1)-X_n(t_2)|,\quad 0<\delta\leq 1.
\end{eqnarray*}
By Theorem 8.2 of the Billingsley's book \cite{Billingsley1968}, to prove the tightness of $(X_n)_{n\geq 1}$, it suffices to show the following two statements.\\

({\bf{I}}): \emph{For each positive $\eta$, there exists an $a$ such that
\begin{eqnarray*}
\mathbb{P}(|X_n(0)|>a)\leq \eta, \quad n\geq 1
\end{eqnarray*}}
and

({\bf{II}}): \emph{For each positive $\varepsilon$ and $\eta$, there exists a $\delta$, with $0<\delta<1$, and an integer $n_0$ such that
\begin{eqnarray*}
\mathbb{P}(w_{X_n}(\delta)\geq\varepsilon)\leq \eta, \quad n\geq n_0.
\end{eqnarray*}}

Note that ({\bf{I}}) is obvious in our case since
\begin{eqnarray*}
\mathbb{P}(X_n(0)=0)=1,\quad n\geq 1.
\end{eqnarray*}
Therefore, it suffices to show ({\bf{II}}) in the sequel. We will rely on the following lemma.
\begin{lem} \label{lem.3.4.5.6} Assume that $M_n$ satisfies the assumptions imposed in Theorem \ref{th.1}. Let $\epsilon$ be some sufficiently small but fixed positive constant. If for any $t_1,t_2\in[0,1]$ satisfying $|t_2-t_1|\geq n^{-1/2-\epsilon}$ there exists
\begin{eqnarray}
\mathbb{E}\left(X_n(t_2)-X_n(t_1)\right)^4\leq C|t_2-t_1|^\alpha \label{7.8.9}
\end{eqnarray}
with some positive constants $C$ and $\alpha>1$ which are both independent of $t_1,t_2$, then ({\bf{II}}) holds.
\end{lem}
\begin{proof}
Note that by definition, we need to show that for any positive $\varepsilon$ and $\eta$, there exists a $\delta\in(0,1)$ and a sufficiently large $n_0$ such that for $n\geq n_0$
\begin{eqnarray*}
\mathbb{P}\left(\sup_{|t_2-t_1|\leq \delta}|X_n(t_2)-X_n(t_1)|\geq \varepsilon\right)\leq \eta.
\end{eqnarray*}
By the discussions in \cite{Billingsley1968} (see (8.12) of \cite{Billingsley1968}), it suffices to show that for $n\geq n_0$ and $0\leq t_1\leq 1$,
\begin{eqnarray*}
\frac{1}{\delta}\mathbb{P}\left(\sup_{t_1\leq t_2\leq t_1+\delta}|X_n(t_2)-X_n(t_1)|\geq \varepsilon/3\right)\leq \eta.
\end{eqnarray*}
Now we set
\begin{eqnarray*}
m=m(n):=\lfloor n^{1/2+\epsilon/2}\rfloor.
\end{eqnarray*}
Note that
\begin{eqnarray*}
0&\leq &\sup_{t_1\leq t_2\leq t_1+\delta}|X_n(t_2)-X_n(t_1)|-\max_{0\leq j\leq m}|X_n(t_1+\frac{j}{m}\delta)-X_n(t_1)|\nonumber\\
&\leq & C\sqrt{n}\max_{1\leq j\leq m}\sum_{i=\lfloor n(t_1+\frac{j-1}{m}\delta)\rfloor}^{\lfloor n(t_1+\frac{j}{m}\delta)\rfloor}(|y_i|^2+\frac1n)
\leq C\sqrt{n}\frac{n\delta}{m} \max_i (|y_i|^2+\frac1n)\leq C\delta n^{-\epsilon/4}
\end{eqnarray*}
with overwhelming probability for sufficiently large $n$. Here in the last inequality above we used the isotropic delocalization property (\ref{p.6}).
Thus it suffices to show that for $m=\lfloor n^{1/2+\epsilon/2}\rfloor$,
\begin{eqnarray}
\frac{1}{\delta}\mathbb{P}\left(\max_{0\leq j\leq m}|X_n(t_1+\frac{j}{m}\delta)-X_n(t_1)|\geq \varepsilon/4\right)\leq \eta/2. \label{4.5.6.7}
\end{eqnarray}
By Theorem 12.2 and the proof of Theorem 12.3 of \cite{Billingsley1968}, to obtain (\ref{4.5.6.7}), it suffices to show that for any $t_1,t_2\in[0,1]$ such that $|t_2-t_1|\geq m^{-1}\delta$  one has
\begin{eqnarray*}
\mathbb{E}\left(X_n(t_2)-X_n(t_1)\right)^4\leq C|t_2-t_1|^\alpha
\end{eqnarray*}
with some positive constants $C$ and $\alpha>1$. When $n$ is sufficiently large, by the definition of $m$, it suffices to have (\ref{7.8.9}) when $t_2-t_1\geq n^{-1/2-\epsilon}$. Thus we complete the proof.
\end{proof}

Below we will verify the condition (\ref{7.8.9}) of Lemma \ref{lem.3.4.5.6}  for $t_1, t_2\in[0,1]$ such that $t_2-t_1\geq n^{-1/2-\epsilon}$. At first, we construct a modified process as
\begin{eqnarray*}
Y_n(F_{sc}^{-1}(t))=\sqrt{\frac{\beta n}{2}}\sum_{i=1}^{nF_n(F_{sc}^{-1}(t))}(|y_i|^2-\frac1n).
\end{eqnarray*}
Here we specify $F_{sc}^{-1}(0)=-2$ and $F_{sc}^{-1}(1)=2$. Note that $Y_n(s), s\in[-2,2]$ is just $X_n(F_n(s))$ restricted on $[-2,2]$. Moreover,
\begin{eqnarray*}
X_n(t)=X_n(F_{sc}(F^{-1}_{sc}(t)))
,\quad t\in[0,1].\end{eqnarray*}
Then by using (\ref{8889}) and (\ref{p.6}) again one obtains that with overwhelming probability,
\begin{eqnarray}
\sup_{t\in[0,1]}|Y_n(F_{sc}^{-1}(t))-X_n(t)|&\leq & \sqrt{n}(\log n)^{C\log\log n} \max_{i}(|y_i|^2+\frac{1}{n})\nonumber\\
&\leq & \frac{(\log n)^{C'\log\log n}}{\sqrt{n}}.
 \label{p.8}
\end{eqnarray}
Moreover, we also have the definite bounds
\begin{eqnarray}
\sup_{t\in[0,1]}|X_n(t)|, |Y_n(F_{sc}^{-1}(t))|\leq \sqrt{n}\sum_{i=1}^n(|y_i|^2+\frac1n)=2\sqrt{n}. \label{p.9}
\end{eqnarray}
Then by combining (\ref{p.8}) and (\ref{p.9}), for $t_2-t_1\geq n^{-1/2-\epsilon}$ we have
\begin{eqnarray}
\mathbb{E}((Y_n(F_{sc}^{-1}(t_2))-Y_n(F_{sc}^{-1}(t_1)))-(X_n(t_2)-X_n(t_1)))^4\leq C(t_2-t_1)^2. \label{0.0.8}
\end{eqnarray}
Therefore, it suffices to show that for any $t_1, t_2\in[0,1]$ satisfying  $t_2-t_1\geq n^{-1/2-\epsilon}$, one has
\begin{eqnarray}
\mathbb{E}(Y_n(F_{sc}^{-1}(t_2))-Y_n(F_{sc}^{-1}(t_1))^4\leq
C(t_2-t_1)^{\alpha} \label{p.10}
\end{eqnarray}
with some positive  constants $C$ and $\alpha>1$.
Now set
\begin{eqnarray*}
s_1=F_{sc}^{-1}(t_1),\quad s_2=F_{sc}^{-1}(t_2).
\end{eqnarray*}
By the explicit formula of the semicircle law (\ref{p.0}), it is elementary to see that when $t_2-t_1\geq n^{-1/2-\epsilon}$, there exists $s_2-s_1\geq Cn^{-1/2-\epsilon}$ for some positive constant $C$. Actually, one has
\begin{eqnarray}
C(t_2-t_1)\leq s_2-s_1\leq C'(t_2-t_1)^{2/3}. \label{10.2}
\end{eqnarray}
holds uniformly in $s_1, s_2$ with some positive constants $C$ and $C'$ .
Thus it suffices to verify that when
\begin{eqnarray*}
s_2-s_1\geq Cn^{-1/2-\epsilon},
\end{eqnarray*}
one has
\begin{eqnarray}
\mathbb{E}(Y_n(s_2)-Y_n(s_1))^4\leq (s_2-s_1)^2\leq C'(t_2-t_1)^{4/3}. \label{p.11}
\end{eqnarray}
Then (\ref{0.0.8}) together with (\ref{p.11}) imply (\ref{p.10}) with $\alpha=4/3$.
Note that, by definition we have
\begin{eqnarray*}
Y_n(s)=\sqrt{\frac{\beta n}{2}}\sum_{i=1}^{nF_n(s)}(|y_i|^2-\frac1n)
=\sqrt{\frac{\beta n}{2}}\sum_{i=1}^{n}(|y_i|^2-\frac1n)\mathbf{1}_{\{\lambda_i\leq s \}}, \quad s\in[-2,2].
\end{eqnarray*}
Thus
\begin{eqnarray*}
Y_n(s_2)-Y_n(s_1)
=\sqrt{\frac{\beta n}{2}}\sum_{i=1}^{n}(|y_i|^2-\frac1n)\mathbf{1}_
{\{s_1<\lambda_i\leq s_2 \}},\quad s_1,s_2\in[-2,2].
\end{eqnarray*}

In the sequel, we will show (\ref{p.11}) by a Green function comparison strategy. To this end, at first, we will approximate the indicator functions $\mathbf{1}_{\{s_1<\lambda_i\leq s_2\}},i=1,\cdots,n$ by smooth functions expressed in terms of the Green function with the help of the following lemma.
\begin{lem} \label{lem.t.1}Under the assumptions of Theorem \ref{th.1}, for $\eta=n^{-1/2-\epsilon}(s_2-s_1)^{1/2}$ with some sufficiently small but fixed positive constant $\epsilon$,  when $s_2-s_1\geq Cn^{-1/2-\epsilon}$ with some  positive constant $C$, we have
\begin{eqnarray*}
\mathbb{E}\left(\sqrt{n}\sum_{i=1}^{n}(|y_i|^2-\frac1n)\left(\mathbf{1}_{\{s_1
<\lambda_i\leq s_2\}}-\frac{1}{\pi}\int_{s_1}^{s_2}\Im\frac{1}{\lambda_i-(E+\sqrt{-1}\eta)}dE\right)\right)^4\leq C'(s_2-s_1)^2.
\end{eqnarray*}
with some positive constant $C'$.
\end{lem}
\begin{proof} By the isotropic delocalization property (\ref{p.6}) and the definite bound $|y_i|^2\leq 1$ we see that it suffices to show for some sufficiently small constant $\epsilon>0$,
\begin{eqnarray}
n^{-2+\epsilon}\mathbb{E}\left(\sum_{i=1}^{n}\left|\mathbf{1}_{\{s_1<\lambda_i\leq s_2\}}-\frac{1}{\pi}\int_{s_1}^{s_2}\Im\frac{1}{\lambda_i-(E+\sqrt{-1}\eta)}\right|\right)^4\leq C'(s_2-s_1)^{2}\label{t.2}
\end{eqnarray}

Now we choose
\begin{eqnarray}
\theta:=n^{-1/2-\epsilon/2}(s_2-s_1)^{1/2}\gg \eta\geq Cn^{-\frac{3}{4}-\frac{3\epsilon}{2}}. \label{p.14}
\end{eqnarray}
Observe that both $\eta$ and $\theta$ are much less than $ s_2-s_1$. Now we split the real line into $\mathbb{R}=\mathbf{L}_1\cup\mathbf{L}_2$, where
\begin{eqnarray*}
\mathbf{L}_1=(-\infty, s_1-\theta)\cup(s_1+\theta, s_2-\theta)\cup(s_2+\theta, \infty),\quad \mathbf{L}_2=\mathbb{R}\setminus\mathbf{L}_1.
\end{eqnarray*}
We will show that when $\lambda_i\in \mathbf{L}_1$, one has
\begin{eqnarray}
\left|\mathbf{1}_{\{s_1<\lambda_i\leq s_2\}}-\frac{1}{\pi}\int_{s_1}^{s_2}\Im\frac{1}{\lambda_i-(E+\sqrt{-1}\eta)}dE\right|\leq C\eta(\frac{1}{|\lambda_i-s_1|}+\frac{1}{|\lambda_i-s_2|}).\label{t.1}
\end{eqnarray}
To see (\ref{t.1}), we use the following elementary fact.
\begin{eqnarray*}
\frac{1}{\pi}\int_{s_1}^{s_2}\Im\frac{1}{\lambda_i-(E+\sqrt{-1}\eta)}dE
&=&\frac{1}{\pi}\int_{s_1}^{s_2}\frac{\eta}{(\lambda_i-E)^2+\eta^2}dE\nonumber\\
&=&\frac{1}{\pi}(\arctan\frac{s_2-\lambda_i}{\eta}-\arctan\frac{s_1-\lambda_i}{\eta}).
\end{eqnarray*}
Note that when $\lambda_i\in \mathbf{L}_1$, then one has
\begin{eqnarray*}
|\lambda_i-s_2|,|\lambda_i-s_1|\geq\theta\gg\eta.
\end{eqnarray*}
By the basic asymptotic properties of $\arctan(x)$ one has for $\lambda_i\in\mathbf{L}_1$,
\begin{eqnarray*}
\left|\mathbf{1}_{\{s_1<\lambda_i\leq s_2\}}-\frac{1}{\pi}(\arctan\frac{s_2-\lambda_i}{\eta}-\arctan\frac{s_1-\lambda_i}{\eta})\right|\leq C\eta\left(\frac{1}{|\lambda_i-s_1|}+\frac{1}{|\lambda_i-s_2|}\right)
\end{eqnarray*}
with some positive constant $C$.
Let $N_n(I)$ be the number of the eigenvalues falling into the region $I\in \mathbb{R}$. Then we have
\begin{eqnarray}
&&\sum_{i=1}^{n}\left|\mathbf{1}_{\{s_1<\lambda_i\leq s_2\}}-\frac{1}{\pi}\int_{s_1}^{s_2}\Im\frac{1}{\lambda_i-(E+i\eta)}\right|\nonumber\\
&\leq &C\left(\eta\sum_{i:\lambda_i\in\mathbf{L_1}}(\frac{1}{|\lambda_i-s_1|}+\frac{1}{|\lambda_i-s_2|})+N_n(\mathbf{L}_2)\right)\label{t.3}
\end{eqnarray}
Now we use the so-called local semicircle law (for instance, see Theorem 1.8 of \cite{TV2010}) in the sense that for any interval $I\in\mathbb{R}$ with its length $|I|\geq n^{-1+c}$ for any sufficiently small but fixed constant $c>0$,
\begin{eqnarray}
N_n(I)=O(n|I|) \label{p.12}
\end{eqnarray}
with overwhelming probability when $n$ is sufficiently large. Now we decompose the real line as \begin{eqnarray*}
\mathbb{R}=(-\infty, -5)\cup\bigg{(}\bigcup_{k=1}
^{K_n}I_k\bigg{)}
\cup(5,+\infty),
\end{eqnarray*}
where $K_n=O(n^{1-c})$ and
\begin{eqnarray*}
I_k=[-5+(k-1)n^{-1+c},-5+kn^{-1+c}].
\end{eqnarray*}
Here we can choose $K_n$ such that
\begin{eqnarray*}
-5+(K_n-1)n^{-1+c}<5,\quad -5+K_nn^{-1+c}\geq 5.
\end{eqnarray*}
We will show that
\begin{eqnarray}
\sum_{i:\lambda_i\in\mathbf{L_1}}(\frac{1}{|\lambda_i-s_1|}+\frac{1}{|\lambda_i-s_2|}) \leq Cn\log  n \label{p.15}
\end{eqnarray}
with overwhelming probability. At first, by the rigidity property provided in \cite{EYY2012} we see that all the eigenvalues of $M_n$ are in $[-5,5]$ with overwhelming probability. Thus it suffices to show with overwhelming probability,
\begin{eqnarray*}
\sum_{i:\lambda_i\in\mathbf{L_1}\cap [-5,5]}(\frac{1}{|\lambda_i-s_1|}+\frac{1}{|\lambda_i-s_2|})
&\leq & \sum_{k=1}^{K_n}\sum_{i: \lambda_i\in \mathbf{L}_1\cap I_k}(\frac{1}{|\lambda_i-s_1|}+\frac{1}{|\lambda_i-s_2|}) \nonumber\\
&\leq &C \sum_{k=1}^{K_n}\sum_{i: \lambda_i\in \mathbf{L}_1\cap I_k} \frac{1}{\theta+(k-1)n^{-1+c}}\nonumber\\
&\leq & C\sum_{k=1}^{K_n} \frac{n|I_k|}{\theta+(k-1)n^{-1+c}}\nonumber\\
&\leq & Cn\log n.
\end{eqnarray*}
Here in the third step we have used (\ref{p.12}). Moreover, by (\ref{p.14}) and (\ref{p.12}) we also have
\begin{eqnarray}
N_n(\mathbf{L}_2)\leq Cn\theta \label{p.16}
\end{eqnarray}
with overwhelming probability for some positive constant $C$. Combining (\ref{t.3}), (\ref{p.15}) and (\ref{p.16})
 we have
\begin{eqnarray}
\sum_{i=1}^{n}\left|\mathbf{1}_{\{s_1<\lambda_i\leq s_2\}}-\frac{1}{\pi}\int_{s_1}^{s_2}\Im\frac{1}{\lambda_i-(E+\sqrt{-1}\eta)}\right|
\leq C(\eta n\log n+n\theta) \label{p.17}
\end{eqnarray}
with overwhelming probability. Now by noticing that the left hand side of (\ref{p.17}) is bounded by $2n$ definitely, we also have
\begin{eqnarray}
\mathbb{E}\left(\sum_{i=1}^{n}\left|\mathbf{1}_{\{s_1<\lambda_i\leq s_2\}}-\frac{1}{\pi}\int_{s_1}^{s_2}\Im\frac{1}{\lambda_i-(E+\sqrt{-1}\eta)}\right|\right)^4\leq  C(\eta n\log n+n\theta)^ 4. \label{p.18}
\end{eqnarray}
Then by (\ref{p.18}) and the definitions of $\theta$ and $\eta$ we can see (\ref{t.2}). Thus we complete the proof.
\end{proof}
Note that
\begin{eqnarray}
&&\sum_{i=1}^{n}(|y_i|^2-\frac1n)\frac{1}{\pi}\int_{s_1}^{s_2}\Im\frac{1}{\lambda_i-(E+\sqrt{-1}\eta)}dE\nonumber\\
&&=\frac{1}{\pi}\int_{s_1}^{s_2}\Im(\mathbf{x}^*G(z)\mathbf{x}-\frac1n trG(z))dE \label{9999}
\end{eqnarray}
For simplicity, we will use the notation in \cite{KY2012} to write $A_{\mathbf{v}\mathbf{w}}=\mathbf{v}^*
A\mathbf{w}$ for any matrix $A$. Particularly, $A_{\mathbf{v}\mathbf{e}_i}$ and $A_{\mathbf{e}_i\mathbf{v}}$ will be simply denoted by $A_{\mathbf{v}i}$ and $A_{i\mathbf{v}}$ in the sequel.

Then with the aid of Lemma \ref{lem.t.1} and (\ref{9999}), it suffices to prove the following lemma.
\begin{lem}  \label{lem.t.0}Let $z=E+\sqrt{-1}\eta $ with $\eta=n^{-1/2-\epsilon}(s_2-s_1)^{1/2}$, where $\epsilon$ is some sufficiently small but fixed positive constant. And we assume that $s_1,s_2\in[-2,2]$ such that $s_2-s_1\geq n^{-1/2-\epsilon}$. Under the assumptions of Theorem \ref{th.1}, one has
\begin{eqnarray}
\mathbb{E}\left(\sqrt{n}\int_{s_1}^{s_2}\Im(G_{\mathbf{x}\mathbf{x}}(z)-\frac1n trG(z))dE\right)^4\leq C(s_2-s_1)^2. \label{t.4}
\end{eqnarray}
\end{lem}
\begin{proof} Below, we will focus on the real case for simplicity. The proof for the complex case is just analogous.
Let
\begin{eqnarray*}
\widetilde{M}_n=\frac{1}{\sqrt{n}}(\tilde{v}_{ij})_{n,n}
\end{eqnarray*}
be GOE and $\widetilde{G}(z)$ be its corresponding Green function. At first, we will show that (\ref{t.4}) holds for GOE. Note that by Lemma \ref{lem.t.1} we can go back to the original quantity to show
\begin{eqnarray}
\mathbb{E}(Y_n(s_2)-Y_n(s_1))^4\leq C(s_2-s_1)^2.
\end{eqnarray}
By (\ref{10.2}) it suffices to show that
\begin{eqnarray*}
\mathbb{E}\left(\sqrt{n}\sum_{i=1}^n(|y_i|^2-\frac1n)\mathbf{1}_{\{s_1< \lambda_i\leq s_2\}}\right)^4\leq C(t_2-t_1)^2.
\end{eqnarray*}
Note that for the Gaussian case, $\mathbf{y}$ is uniformly distributed on $S^{n-1}$.
By using (\ref{8889}) and the isotropic delocalization property (\ref{p.6}) again,  we see that it suffices to verify
\begin{eqnarray}
\mathbb{E}\left(\sqrt{n}\sum_{i=\lfloor nt_1\rfloor}^{\lfloor nt_2\rfloor}(|y_i|^2-\frac1n)\right)^4\leq C(t_2-t_1)^2. \label{10.3}
\end{eqnarray}
Recall the fact
\begin{eqnarray*}
\mathbf{y}\xlongequal{d}\frac{\mathbf{g}_{\mathbb{R}}}{||\mathbf{g}_{\mathbb{R}}||}.
\end{eqnarray*}
Here $\mathbf{g}_{\mathbb{R}}=(g_1,\cdots,g_n)^T$ is the $n\times 1$ random vector with i.i.d $N(0,1)$ coefficients. Note that it is elementary to see
\begin{eqnarray}
\mathbb{E}|y_i^2|^m=\mathbb{E}{\frac{g_i^{2m}}{||\mathbf{g}_\mathbb{R}||^{2m}}}=O(n^{-m}) \label{p.177}
\end{eqnarray}
for any given integer $m\geq 0$.
Moreover, we have the following lemma.
\begin{lem}Assume that $\mathbf{y}=(y_1,\cdots, y_n)^T$ is uniformly distributed on $S^{n-1}$. For any $i,j,k,l$ different from each other,
\begin{eqnarray}
\mathbb{E}(y_i^2-\frac{1}{n})(y_j^2-\frac1n)(y_k^2-\frac1n)^2=O(n^{-5}),\label{p.166}
\end{eqnarray}
\begin{eqnarray}
\mathbb{E}(y_i^2-\frac1n)(y_j^2-\frac1n)(y_k^2-\frac1n)(y_l^2-\frac1n)=O(n^{-6}). \label{p.188}
\end{eqnarray}
\end{lem}
\begin{proof} Note that we always have
\begin{eqnarray*}
\sum_{i=1}^n(y_i^2-\frac1n)=0.
\end{eqnarray*}
Therefore, we have
\begin{eqnarray}
0&=&\mathbb{E}(\sum_{i=1}^n(y_i^2-\frac1n))(\sum_{j=1}^n(y_j^2-\frac1n))(\sum_{k=1}^n(y_k^2-\frac1n)^2)\nonumber\\
&=&\sum_{i,k}\mathbb{E}(y_i^2-\frac1n)^2(y_k^2-\frac1n)^2+2\sum_{i\neq k}\mathbb{E}(y_i^2-\frac1n)(y_k^2-\frac1n)^3\nonumber\\
&&+\sum_{i,j,k\text{ are mutually distinct}}\mathbb{E}(y_i^2-\frac1n)(y_j^2-\frac1n)(y_k^2-\frac1n)^2\nonumber\\
&=&\sum_{i,j,k\text{ are mutually distinct}}\mathbb{E}(y_i^2-\frac1n)(y_j^2-\frac1n)(y_k^2-\frac1n)^2+O(n^{-2}), \label{p.19}
\end{eqnarray}
where the last step follows from (\ref{p.177}) directly. Now note that $\mathbf{y}$ is an exchangeable random vector. Thus by symmetry, we see that every term in the summation in (\ref{p.19}) is the same, thus (\ref{p.166}) follows. (\ref{p.188}) can be verified similarly. Note
\begin{eqnarray*}
0&=&\mathbb{E}(\sum_{i=1}^n(y_i^2-\frac1n))(\sum_{j=1}^n(y_j^2-\frac1n))(\sum_{k=1}^n(y_k^2-\frac1n))(\sum_{l=1}^n(y_l^2-\frac1n))\nonumber\\
&=& \sum_{i}
\mathbb{E}(y_i^2-\frac1n)^4 +4\sum_{i\neq j}
\mathbb{E}(y_i^2-\frac1n)^3(y_j^2-\frac1n)\nonumber\\
&&+3\sum_{i\neq j}
\mathbb{E}(y_i^2-\frac1n)^2(y_j^2-\frac1n)^2\nonumber\\
&&+6\sum_{
i,j,k\text{ are mutually distinct}}\mathbb{E}(y_i^2-\frac1n)(y_j^2-\frac1n)(y_k^2-\frac1n)^2\nonumber\\
&&+\sum_{
i,j,k,l\text{ are mutually distinct}}
\mathbb{E}(y_i^2-\frac1n)(y_j^2-\frac1n)(y_k^2-\frac1n)(y_l^2-\frac1n)\nonumber\\
&&=\sum_{
i,j,k,l\text{ are mutually distinct}}
\mathbb{E}(y_i^2-\frac1n)(y_j^2-\frac1n)(y_k^2-\frac1n)(y_l^2-\frac1n)+O(n^{-2}),
\end{eqnarray*}
where the last step follows from (\ref{p.177}) and (\ref{p.166}). Now again by symmetry, we can get (\ref{p.188}). So we complete the proof.
\end{proof}
Now by using (\ref{p.177}), (\ref{p.166}) and (\ref{p.188}) we have
\begin{eqnarray*}
&&\mathbb{E}\left(\sum_{i=\lfloor nt_1\rfloor}^{\lfloor nt_2\rfloor}(y_i^2-\frac1n)\right)^4\nonumber\\
&=&\sum_{i=
\lfloor nt_1\rfloor}^{\lfloor nt_2\rfloor}\mathbb{E}(y_i^2-\frac1n)^4+4\sum_{i\neq j=
\lfloor nt_1\rfloor}^{\lfloor nt_2\rfloor}
\mathbb{E}(y_i^2-\frac1n)^3(y_j^2-\frac1n)\nonumber\\
&&+3\sum_{i\neq j=
\lfloor nt_1\rfloor}^{\lfloor nt_2\rfloor}
\mathbb{E}(y_i^2-\frac1n)^2(y_j^2-\frac1n)^2\nonumber\\
&&+6\sum_{\substack{i,j,k=\lfloor nt_1\rfloor\\
i,j,k\text{ are mutually distinct}}}^{\lfloor nt_2\rfloor}\mathbb{E}(y_i^2-\frac1n)(y_j^2-\frac1n)(y_k^2-\frac1n)^2\nonumber\\
&&+\sum_{\substack{i,j,k,l=
\lfloor nt_1\rfloor\\
i,j,k,l\text{ are mutually distinct}}}^{\lfloor nt_2\rfloor}
\mathbb{E}(y_i^2-\frac1n)(y_j^2-\frac1n)(y_k^2-\frac1n)(y_l^2-\frac1n)\nonumber\\
&\leq &Cn^{-2}(t_2-t_1)^2,
\end{eqnarray*}
which implies (\ref{10.3}) for GOE. Moreover, by the discussions above, (\ref{t.4}) for the Gaussian case follows.

Therefore, it remains to compare the general case with the Gaussian case. That is to say,
we only need to show that
\begin{eqnarray}
&&\bigg{|}\mathbb{E}\left(\int_{s_1}^{s_2}\Im(G_{\mathbf{x}\mathbf{x}}(z)-\frac1n trG(z))dE\right)^4-\mathbb{E}\left(\int_{s_1}^{s_2}\Im(\widetilde{G}_{\mathbf{x}\mathbf{x}}(z)-\frac1n tr\widetilde{G}(z))dE\right)^4\bigg{|}\nonumber\\
&&\leq Cn^{-2}(s_2-s_1)^2. \label{10.4}
\end{eqnarray}

To simplify the discussions in the sequel, we truncate the matrix elements of $M_n$ and $\widetilde{M}_n$ at $n^{\epsilon}$, where $\epsilon>0$ is the small constant chosen in Lemma \ref{lem.t.0}. By Condition \ref{con.1}, we see that the truncated matrices coincide with the original ones with overwhelming probabilities. Thus their corresponding Green functions also equal to the original ones with overwhelming probabilities. Following from this fact and the definite bounds
\begin{eqnarray*}
||G(z)||_{op},||\widetilde{G}(z)||_{op}\leq \eta^{-1},
\end{eqnarray*}
we see it suffices to prove (\ref{10.4}) for the truncated matrices. Therefore, below we will make the additional assumption of
\begin{eqnarray*}
\max_{ij}|v_{i,j}|\leq n^{\epsilon},\quad \max_{ij}|\tilde{v}_{i,j}|\leq n^{\epsilon}.
\end{eqnarray*}
Note that the truncation may change the first four moments of the original elements by tiny amounts. Actually, such minor changes are smaller than $n^{-K}$ with any positive constant $K$ when $n$ is sufficiently large. It will be clear that such small changes on the moments of elements do not affect our comparison procedure. So for simplicity, we will still regard that the two truncated matrices matches to the first four moments. Moreover,  note that all the results needed from the references such as \cite{EYY2012} and \cite{KY2012} hold with overwhelming probabilities for the original matrices, while the truncated matrices equal to the original ones with overwhelming probabilities, thus the results from these references are still valid for the truncated matrices.

The main idea to show (\ref{10.4}) is a Green function comparison strategy based on the discussions in \cite{KY2012}. To pursue this approach, we need to introduce some notation. At first we assign a bijective ordering map $\phi$ on the index set of the matrix elements,
\begin{eqnarray*}
\phi: \{(i,j):1\leq i\leq j\leq n\}\rightarrow
\{1,\cdots,\frac{n(n+1)}{2}\}.
\end{eqnarray*}
For $1\leq \gamma\leq n(n+1)/2$, we define the matrix $M_n^{\gamma}$ to be the Wigner matrix with its $(i,j)$ element being $v_{ij}/\sqrt{n}$ if $\phi(i,j)\leq \gamma$ or $\tilde{v}_{ij}/\sqrt{n}$ otherwise. Correspondingly, we denote the Green function of $M_n^\gamma$ by $G^{\gamma}(z)$. Thus it suffices to estimate the one step difference
\begin{eqnarray*}
\bigg{|}\mathbb{E}\left(\int_{s_1}^{s_2}\Im(G_{\mathbf{x}\mathbf{x}}^{\gamma}(z)-\frac1n trG^{\gamma}(z))dE\right)^4-\mathbb{E}\left(\int_{s_1}^{s_2}\Im(G_{\mathbf{x}\mathbf{x}}^{\gamma-1}(z)-\frac1n tr{G}^{\gamma-1}(z))dE\right)^4\bigg{|}
\end{eqnarray*}
and in the end we will use a telescoping argument to sum up all these one step differences to obtain (\ref{10.4}).

Now without loss of generality, we assume that $\gamma=\phi(a,b)$. Thus $M_n^{\gamma}$ and $M_n^{\gamma-1}$ only differ in the $(a,b)$ and $(b,a)$-th elements. Let
\begin{eqnarray*}
E^{ab}=\mathbf{e}_a\mathbf{e}_b^*.
\end{eqnarray*}
Then we can write
\begin{eqnarray*}
M_n^{\gamma}=Q+n^{-1/2}V,\quad M_n^{\gamma-1}=Q+n^{-1/2}\widetilde{V},
\end{eqnarray*}
where
\begin{eqnarray*}
&&V=(1-\delta_{ab}/2)(v_{ab}E^{ab}+v_{ba}E^{ba}),\nonumber\\
&&\widetilde{V}=(1-\delta_{ab}/2)(\tilde{v}_{ab}E^{ab}+\tilde{v}_{ba}E^{ba}),
\end{eqnarray*}
and then $Q$ is a random matrix independent of $v_{ab}$ and $\tilde{v}_{ab}$.
Let
\begin{eqnarray*}
R(z):=(Q-z)^{-1}.
\end{eqnarray*}
Moreover, for simplicity, we rewrite $G^{\gamma-1}(z)$ and $G^{\gamma}(z)$ as $S(z)$ and $T(z)$ respectively.
And when there is no confusion, we will omit the variable $z$ from the above notation. Now
by the resolvent expansion, one has
\begin{eqnarray}
S=R+\sum_{k=1}^{4}n^{-k/2}(-RV)^kR+n^{-5/2}(-RV)^5S.\label{t.7}
\end{eqnarray}
 Then we can write
\begin{eqnarray*}
&&\mathbb{E}\left(\int_{s_1}^{s_2}\Im(S_{\mathbf{x}
\mathbf{x}}(z)-n^{-1}trS(z))dE
\right)^4\nonumber\\
&=&\mathbb{E}\left(\int_{s_1}^{s_2}\Im\bigg{
(}(R_{\mathbf{x}\mathbf{x}}(z)-n^{-1}trR(z))\right.\nonumber\\
&&+\sum_{k=1}^{4}n^{-k/2}([(-R(z)V)^kR(z)]_{\mathbf{x}\mathbf{x}}-n^{-1}tr(-R(z)V)^kR(z))\nonumber\\
&&+\left.n^{-5/2}([(-R(z)V)^5S(z)]_{\mathbf{x}\mathbf{x}}-n^{-1}tr(-R(z)V)^5S(z))\bigg{)}dE\right)^4\nonumber\\
&=:&\mathbb{E}(\Im(\mathcal{F}_0
+\sum_{k=1}^{4}n^{-k/2}\mathcal{F}_k+n^{-5/2}\mathcal{F}_5))^4.
\end{eqnarray*}
Here $\mathcal{F}_i:=\mathcal{F}_i(a,b),i=0\cdots,5$ whose definitions are given by
\begin{eqnarray*}
&&\mathcal{F}_0:=\int_{s_1}^{s_2}(R_{\mathbf{x}\mathbf{x}}(z)-n^{-1}trR(z))dE,\nonumber\\
&&\mathcal{F}_k:=\int_{s_1}^{s_2}([(-R(z)V)^kR(z)]_{\mathbf{x}\mathbf{x}}-n^{-1}tr(-R(z)V)^kR(z))dE,\quad k=1,\cdots, 4,\nonumber\\
&&\mathcal{F}_5:=\int_{s_1}^{s_2}([(-R(z)V)^kS(z)]_{\mathbf{x}\mathbf{x}}-n^{-1}tr(-R(z)V)^kS(z))dE.
\end{eqnarray*}

Observe that in the real case, every $[(RV)^kR]_{\mathbf{x}\mathbf{x}}$ can be written as a summation of the terms in the form of
\begin{eqnarray*}
(v_{ab})^{k}\mathbf{q}_{k,a,b}(R,\mathbf{x}),
\end{eqnarray*}
where $\mathbf{q}_{k,a,b}(R,\mathbf{x})$ is some product of the factors $R_{\mathbf{x}a}$, $R_{a\mathbf{x}}$, $R_{\mathbf{x}b}$, $R_{b\mathbf{x}}$, $R_{aa}$, $R_{ab}$, $R_{ba}$ and $R_{bb}$. We remind here in the complex case, $(v_{ab})^k$ should be replaced by $(v_{ab})^{k_1}(v_{ba})^{k_2}$ with $k_1+k_2=k$. Moreover, the total number  of the factors $R_{\mathbf{x}a}$, $R_{a\mathbf{x}}$, $R_{\mathbf{x}b}$, $R_{b\mathbf{x}}$ in every $\mathbf{q}_{k,a,b}(R,\mathbf{x})$ is $2$. For example,
\begin{eqnarray*}
[(RV)^2R]_{\mathbf{x}\mathbf{x}}=(v_{ab})^2\left(R_{\mathbf{x}a}R_{bb}R_{a\mathbf{x}}
+R_{\mathbf{x}b}R_{ab}R_{a\mathbf{x}}
+R_{\mathbf{x}a}R_{ba}R_{b\mathbf{x}}+
R_{\mathbf{x}b}R_{aa}R_{b\mathbf{x}}\right).
\end{eqnarray*}

In the following Lemma \ref{lem.t.2} we will give some crude bounds for the quantities $\mathcal{F}_k$. These bounds will be used to provide a crude bound of
\begin{eqnarray}
\mathbb{E}\left(\int_{s_1}^{s_2}\Im(G_{\mathbf{x}\mathbf{x}}(z)-\frac{1}{n}trG(z))dE\right)^4 \label{1010}
\end{eqnarray}
through a comparison procedure. Then the crude bound for (\ref{1010}) will imply an improved bound for $\mathcal{F}_0$. Such an improved bound combined with another round of comparison can help us to obtain a good bound for (\ref{1010}). In other words, our main route in the sequel is to use a ``bootstrap'' strategy to get a crude bound of (\ref{1010}) at first and then use the crude bound to get the final bound in Lemma \ref{lem.t.0}. For ease of presentation, we will use the notation in \cite{KY2012} to set
\begin{eqnarray*}
\Psi(z):=\sqrt{\frac{\Im m_{sc}(z)}{n\eta}}+\frac{1}{n\eta}.
\end{eqnarray*}
\begin{lem} \label{lem.t.2}Under the assumptions in Lemma \ref{lem.t.0}, one has with overwhelming probability
\begin{eqnarray}
|\mathcal{F}_0(z)|\leq n^{O(\epsilon)}\sup_{E\in[s_1,s_2]}\Psi(z)(s_2-s_1), \label{t.5}
\end{eqnarray}
and for $1\leq k\leq 5$,
\begin{eqnarray}
|\mathcal{F}_k(z)|\leq n^{O(\epsilon)}\sup_{E\in[s_1,s_2]}\left(\left(\frac{\Im S_{\mathbf{x}\mathbf{x}}(z)}{n\eta}+\Psi^2(z)\right)+(|x_a|^2+|x_b|^2)\right)(s_2-s_1).\nonumber\\
\label{t.6}
\end{eqnarray}
\end{lem}
\begin{proof} First of all, by truncation, we have that the elements are bounded by $n^{\epsilon}$. Moreover, by assumption one has
\begin{eqnarray}
2n^{-1/2-\epsilon}\geq \eta\gg n^{-3/4-2\epsilon} \label{p.p.p}
\end{eqnarray}
for some sufficiently small $\epsilon>0$. Thus we always have
$z=E+\sqrt{-1}\eta\in\mathbf{S}$ which is defined in (\ref{101010}).
Now we come to verify (\ref{t.5}). By definition,
\begin{eqnarray*}
\mathcal{F}_0=\int_{s_1}^{s_2}(R_{\mathbf{x}\mathbf{x}}(z)-n^{-1}trR(z))dE.
\end{eqnarray*}
Observe that
\begin{eqnarray*}
n^{\epsilon}\gg (\log n)^{\log\log n}
\end{eqnarray*}
for sufficiently large $n$.
Using the isotropic local semicircle law (\ref{p.5}), one has with overwhelming probability
\begin{eqnarray}
|S_{\mathbf{x}\mathbf{x}}(z)-m_{sc}(z)|\leq n^{O(\epsilon)}\Psi(z),\quad |n^{-1}trS(z)-m_{sc}(z)|\leq n^{O(\epsilon)}\Psi(z).\nonumber\\
\label{p.22}
\end{eqnarray}
 Moreover, by (3.28), (3.29) of \cite{KY2012} and the discussions above them we know
\begin{eqnarray}
|S_{\mathbf{x}a}(z)|,|R_{\mathbf{x}a}(z)|\leq n^{O(\epsilon)}\left(\sqrt{\frac{\Im S_{\mathbf{x}\mathbf{x}}(z)}{n\eta}}+\Psi(z)+|x_a|\right) \label{p.20}
\end{eqnarray}
and
\begin{eqnarray}
|S_{ij}(z)|, |R_{ij}(z)|=O(1) \label{p.21}
\end{eqnarray}
with overwhelming probability. Moreover, analogously, the bound in (\ref{p.20}) also holds for $S_{a\mathbf{x}}(z)$, $R_{a\mathbf{x}}(z)$ with overwhelming probability.
Then by (\ref{t.7}), (\ref{p.22})-(\ref{p.21}) one can get that
\begin{eqnarray*}
|R_{\mathbf{x}\mathbf{x}}(z)-m_{sc}(z)|\leq n^{O(\epsilon)}\Psi(z),\quad |n^{-1}trR(z)-m_{sc}(z)|\leq  n^{O(\epsilon)}\Psi(z).
\end{eqnarray*}
Thus (\ref{t.5}) follows immediately.

Now we come to verify (\ref{t.6}). Note that for $1\leq k\leq 5$, the total number of the factors $R_{\mathbf{x}a}$, $R_{\mathbf{x}b}$, $R_{a\mathbf{x}}$,
$R_{b\mathbf{x}}$, $S_{\mathbf{x}a}$,$S_{\mathbf{x}b}$,$S_{a\mathbf{x}}$ and
$S_{b\mathbf{x}}$ in each $[(-RV)^kR]_{\mathbf{x}\mathbf{x}}$ or $[(-RV)^kS]_{\mathbf{x}\mathbf{x}}$ is $2$. Now
let $p$ be the total number of the factors $R_{\mathbf{x}a}$, $R_{a\mathbf{x}}$,
$S_{\mathbf{x}a}$ and $S_{a\mathbf{x}}$,
and $q$ be the total number of the factors $R_{\mathbf{x}b}$, $R_{b\mathbf{x}}$,
 $S_{\mathbf{x}b}$ and $S_{b\mathbf{x}}$. Thus $p+q=2$. Then by (\ref{p.20}) and (\ref{p.21}) one obtains
\begin{eqnarray*}
|\mathcal{F}_k|&\leq &n^{O(\epsilon)}\sup_{E\in[s_1,s_2]} \sum_{p+q=2}(\sqrt{\frac{\Im S_{\mathbf{x}\mathbf{x}}(z)}{n\eta}}+\Psi(z)+|x_a|)^p\nonumber\\
&&\times(\sqrt{\frac{\Im S_{\mathbf{x}\mathbf{x}}(z)}{n\eta}}+\Psi(z)+|x_b|)^q(s_2-s_1)\nonumber\\
&\leq & n^{O(\epsilon)} \sup_{E\in[s_1,s_2]}(\sqrt{\frac{\Im S_{\mathbf{x}\mathbf{x}}(z)}{n\eta}}+\Psi(z)+|x_a|)^2(s_2-s_1)\nonumber\\
&&+ n^{O(\epsilon)}\sup_{E\in[s_1,s_2]}(\sqrt{\frac{\Im S_{\mathbf{x}\mathbf{x}}(z)}{n\eta}}+\Psi(z)+|x_b|)^2(s_2-s_1)\nonumber\\
&\leq & n^{O(\epsilon)}\sup_{E\in[s_1,s_2]}\left(\left(\frac{\Im S_{\mathbf{x}\mathbf{x}}(z)}{n\eta}+\Psi^2(z)\right)+(|x_a|^2+|x_b|^2)\right)(s_2-s_1)
\end{eqnarray*}
with overwhelming probability.
Thus we conclude the proof of Lemma \ref{lem.t.2}.
\end{proof}
Following from Lemma \ref{lem.t.2} and the definite bound
\begin{eqnarray*}
|\mathcal{F}_0|\leq \eta^{-1}(s_2-s_1),
\end{eqnarray*}
one can see that
\begin{eqnarray*}
\mathbb{E}(\Im\mathcal{F}_0)^4\leq n^{O(\epsilon)}\sup_{E\in[s_1,s_2]}\left(\left(\sqrt{\frac{\Im m_{sc}(z)}{n\eta}}+\frac{1}{n\eta}\right)\right)^4(s_2-s_1)^4\leq n^{-1+O(\epsilon)}(s_2-s_1)^2.
\end{eqnarray*}
Here we used the definition of $\eta$ in Lemma \ref{lem.t.0}. However, relying on Lemma \ref{lem.t.2}, we can provide a better bound on the 4-th moment of $\Im\mathcal{F}_0$ by a ``bootstrap'' strategy. Precisely, we will show the following bound,
\begin{eqnarray}
\mathbb{E}(\Im\mathcal{F}_0)^4\leq n^{-3/2+O(\epsilon)}(s_2-s_1)^2. \label{p.255}
\end{eqnarray}
To verify (\ref{p.255}), we need the following crude bound for
(\ref{1010})
for any Wigner matrix $M_n$ satisfying the assumptions of Theorem \ref{th.1}.
\begin{lem}  \label{lem.t.3} Under the assumptions of Lemma \ref{lem.t.0}, we have
\begin{eqnarray*}
\mathbb{E}\left(\int_{s_1}^{s_2}(\Im(G_{\mathbf{x}\mathbf{x}}(z)-\frac1n trG(z)))dE\right)^4\leq n^{-3/2+O(\epsilon)}(s_2-s_1)^2.
\end{eqnarray*}
\end{lem}
Before proving Lemma (\ref{lem.t.3}), we explain here how it implies (\ref{p.255}).
Note that by definition,
\begin{eqnarray*}
\Im\mathcal{F}_0=\int_{s_1}^{s_2}\Im(R_{\mathbf{x}\mathbf{x}}(z)-n^{-1}trR(z))dE.
\end{eqnarray*}
By (\ref{t.7}), it suffices to show
\begin{eqnarray}
\mathbb{E}\left(\int_{s_1}^{s_2}\Im(S_{\mathbf{x}\mathbf{x}}(z)-n^{-1}trS(z))dE\right)^4\leq n^{-3/2+O(\epsilon)}(s_2-s_1)^2. \label{t.8}
\end{eqnarray}
Note $M_n^{\gamma}$ is also a Wigner matrix satisfying the asumptions in Theorem \ref{th.1}. Thus (\ref{t.8}) follows immediately. So does (\ref{p.255}). Now we come to prove Lemma \ref{lem.t.3}.
\begin{proof} [Proof of Lemma \ref{lem.t.3}]
Now we will use the mentioned strategy to estimate the one step difference
\begin{eqnarray*}
\bigg{|}\mathbb{E}\left(\int_{s_1}^{s_2}(\Im(G_{\mathbf{x}\mathbf{x}}^\gamma(z)-\frac1n trG^{\gamma}(z)))dE\right)^4-\mathbb{E}\left(\int_{s_1}^{s_2}(\Im(G_{\mathbf{x}\mathbf{x}}^{\gamma-1}(z)-\frac1n trG^{\gamma-1}(z)))dE\right)^4\bigg{|}
\end{eqnarray*}
by inserting the crude bounds of $\mathcal{F}_0$ and $\mathcal{F}_k$ provided in Lemma \ref{lem.t.2} at first, and then using the telescoping argument we will obtain the bound in Lemma {\ref{lem.t.3}}. By the discussions above, we have
\begin{eqnarray}
&&\mathbb{E}\left(\Im\int_{s_1}^{s_2}(S_{\mathbf{x}\mathbf{x}}(z)-n^{-1}trS(z))dE\right)^4\nonumber\\
&= &
\mathbb{E}(\Im(\mathcal{F}_0
+\sum_{k=1}^{4}n^{-k/2}\mathcal{F}_k+n^{-5/2}\mathcal{F}_5))^4\nonumber\\
&=&\mathcal{A}_{ab}+\mathbb{E}\sum_{k_1,\cdots, k_4=0}^5n^{-\sum_{i=1}^4k_i/2}\mathbf{1}_{\{\sum_{i=1}^4k_i\geq 5\}}\prod_{i=1}^4\Im\mathcal{F}_{k_i}, \label{t.10}
\end{eqnarray}
where $\mathcal{A}_{ab}$ is a quantity only depending on the first four moments of $v_{ab}$ and independent of $v_{ab}$ itself. Analogously,
\begin{eqnarray}
&&\mathbb{E}\left(\Im\int_{s_1}^{s_2}(T_{\mathbf{x}\mathbf{x}}(z)-n^{-1}trT(z))dE\right)^4\nonumber\\
&=&\mathcal{A}_{ab}+\mathbb{E}\sum_{k_1,\cdots, k_4=0}^5n^{-\sum_{i=1}^4k_i/2}\mathbf{1}_{\{\sum_{i=1}^4k_i\geq 5\}}\prod_{i=1}^4\Im\mathcal{F}_{k_i}(\gamma-1\rightarrow \gamma), \label{10101010}
\end{eqnarray}
where $\mathcal{F}_{k}(\gamma-1\rightarrow \gamma)$ stands for the quantity defined through replacing $V$ and $S$ by $\widetilde{V}$ and $T$ respectively in the definition of $\mathcal{F}_k$. Therefore, to get the one step difference one only needs to estimate the second terms of (\ref{t.10}) and (\ref{10101010}). We will only handle (\ref{t.10}) below. (\ref{10101010}) is just the same.

Note that it suffices to estimate the contribution of
\begin{eqnarray}
\mathbb{E}n^{-\kappa/2}\sum_{k_1,\cdots, k_4=0}^5\mathbf{1}_{\{\sum_{i=1}^4k_i= \kappa\}}\prod_{i=1}^4\Im\mathcal{F}_{k_i}. \label{t.10.1}
\end{eqnarray}
for $5\leq \kappa\leq 20$.
Now we denote the right hand sides of (\ref{t.5}) and (\ref{t.6}) by $D_1$ and $D_2$ respectively. Note that by the assumption on $\eta$,
\begin{eqnarray*}
&&D_1\leq n^{-1/4+O(\epsilon)}(s_2-s_1)^{3/4},
\end{eqnarray*}
and
\begin{eqnarray}
D_2\leq n^{-1/2+O(\epsilon)}(s_2-s_1)^{1/2}+n^{O(\epsilon)}(|x_a|^2+|x_b|^2)(s_2-s_1). \label{p.26}
\end{eqnarray}
with overwhelming probability.
Combining with the definite bounds
\begin{eqnarray}
||R(z)||_{op}, ||S(z)||_{op}\leq \eta^{-1}, \label{p.266}
\end{eqnarray}
 we can use the upper bound in (\ref{p.26}) as a definite one when we take expectations towards the polynomials in $D_2$.
Now let $m$ be the number of $0$ in the collection $\{k_1,k_2,k_3,k_4\}$. Then we have
\begin{eqnarray*}
&&|n^{-\kappa/2}\mathbb{E}\sum_{k_1,\cdots, k_4=0}^5\mathbf{1}_{\{\sum_{i=1}^4k_i= \kappa\}}\prod_{i=1}^4\Im\mathcal{F}_{k_i}|\nonumber\\
&\leq &Cn^{-\kappa/2}\sum_{m=0}^3D_1^m\mathbb{E}D_2^{4-m}\nonumber\\
&\leq &Cn^{-\kappa/2}(s_2-s_1)^2\sum_{m=0}^3\left(n^{-1/4+O(\epsilon)}\right)^m
\left(n^{-1/2+O(\epsilon)}+n^{O(\epsilon)}(|x_a|^2+|x_b|^2)\right)^{4-m}\nonumber\\
&\leq & n^{-\kappa/2}(s_2-s_1)^2\sum_{m=0}^3\left(n^{-m/4+O(\epsilon)}\right)
\left(n^{-(4-m)/2+O(\epsilon)}+n^{O(\epsilon)}(|x_a|^2+|x_b|^2)\right).
\end{eqnarray*}
Here we have used the fact that $|x_a|\leq 1$.
Then it is not difficult to see
\begin{eqnarray}
&&|\sum_{a,b}n^{-\kappa/2}\mathbb{E}\sum_{k_1,\cdots, k_4=0}^5\mathbf{1}_{\{\sum_{i=1}^4k_i= \kappa\}}\prod_{i=1}^4\Im\mathcal{F}_{k_i}|\nonumber\\
&\leq & \sum_{a,b}n^{-\kappa/2}\left(n^{-5/4+O(\epsilon)}+n^{O(\epsilon)}(|x_a|^2+|x_b|^2)\right)(s_2-s_1)^2\nonumber\\
&\leq & n^{-\kappa/2+1+O(\epsilon)}(s_2-s_1)^2. \label{t.9}
\end{eqnarray}
Here we have used the fact that $\sum_a |x_a|^2=1$.
Then by (\ref{t.10}), (\ref{t.9}) and the fact $\kappa\geq 5$ one has
\begin{eqnarray*}
&&|\mathbb{E}\left(\int_{s_1}^{s_2}(\Im(G_{\mathbf{x}\mathbf{x}}(z)-\frac1n trG(z)))dE\right)^4-\mathbb{E}\left(\int_{s_1}^{s_2}(\Im(\widetilde{G}_{\mathbf{x}\mathbf{x}}(z)-\frac1n tr\widetilde{G}(z)))dE\right)^4|\nonumber\\
&&\leq n^{-3/2+O(\epsilon)}(s_2-s_1)^2.
\end{eqnarray*}
Using Lemma \ref{lem.t.0} for the Gaussian case which has been proved above, we can conclude the proof.
\end{proof}
Now we proceed to the proof of Lemma \ref{lem.t.0}. Again we will resort to the telescoping argument. But now we will use (\ref{p.255}) instead of the crude bound (\ref{t.5}). We go back to (\ref{t.10}) and provide a better bound for (\ref{t.10.1}) below. Note that by (\ref{t.9}), it suffices to consider the cases of $\kappa=5$ and $\kappa=6$.

At first, we come to deal with the case of $\kappa=5$. To this end, we split the set $\{\{k_1,\cdots,k_4\}:\sum_{i=1}^4k_i=5\}$ into the following four cases.\\

({\bf{i}}): $\{k_1,k_2,k_3,k_4\}=\{0,0,0,5\}$, \\

({\bf{ii}}): $\{k_1,k_2,k_3,k_4\}=\{0,0,1,4\}, \{0,0,2,3\}$,\\

({\bf{iii}}): $\{k_1,k_2,k_3,k_4\}=\{0,1,1,3\}, (0,1,2,2)$,\\

({\bf{iv}}): $\{k_1,k_2,k_3,k_4\}=\{1,1,1,2\}$.

Note that for case ({\bf{i}}), by using (\ref{p.266}), (\ref{p.255}) and Lemma \ref{lem.t.2} we see that
\begin{eqnarray*}
&&n^{-5/2}\sum_{a,b}|\mathbb{E}\Im\mathcal{F}_{k_1}\Im\mathcal{F}_{k_2}
\Im\mathcal{F}_{k_3}\Im\mathcal{F}_{k_4}|\nonumber\\
&\leq& n^{-5/2}\sum_{a,b}\prod_{i=1}^4(\mathbb{E}(\Im\mathcal{F}_{k_i})^4)^{1/4}\nonumber\\
&\leq &n^{-5/2}\sum_{a,b}n^{-9/8+O(\epsilon)}\left(n^{-1/2+O(\epsilon)}+n^{O(\epsilon)}(|x_a|^2+|x_b|^2)\right)(s_2-s_1)^2\nonumber\\
&=&n^{-17/8+O(\epsilon)}(s_2-s_1)^2.
\end{eqnarray*}
Now we come to evaluate the case ({\bf{ii}}).
\begin{eqnarray*}
&&n^{-5/2}\sum_{a,b}|\mathbb{E}\Im\mathcal{F}_{k_1}\Im\mathcal{F}_{k_2}
\Im\mathcal{F}_{k_3}\Im\mathcal{F}_{k_4}|\nonumber\\
&\leq& n^{-5/2}\sum_{a,b}n^{-3/4+O(\epsilon)}\left(n^{-1/2+O(\epsilon)}+n^{O(\epsilon)}(|x_a|^2+|x_b|^2)\right)^2(s_2-s_1)^2\nonumber\\
&\leq &n^{-5/2}\sum_{a,b}n^{-3/4+O(\epsilon)}\left(n^{-1+O(\epsilon)}+n^{O(\epsilon)}(|x_a|^2+|x_b|^2)\right)(s_2-s_1)^2\nonumber\\
&=&n^{-9/4+O(\epsilon)}(s_2-s_1)^2.
\end{eqnarray*}
Thus both of these two cases can be bounded well by our estimates in (\ref{t.6}) and (\ref{p.255}). Now we come to deal with the cases ({\bf{iii}}) and ({\bf{iv}}) whose estimates need more accurate bounds on the products of $\Im\mathcal{F}_k$. This relies on the observation that with overwhelming probability,
\begin{eqnarray}
|\mathcal{F}_1|&\leq& n^{O(\epsilon)}\sup_{E\in[s_1,s_2]}(\sqrt{\frac{\Im S_{\mathbf{x}\mathbf{x}}(z)}{n\eta}}+\Psi(z)+|x_a|)(\sqrt{\frac{\Im S_{\mathbf{x}\mathbf{x}}(z)}{n\eta}}+\Psi(z)+|x_b|)(s_2-s_1)\nonumber\\
&\leq &\left(n^{-1/2+O(\epsilon)}+n^{-1/4+O(\epsilon)}(|x_a|+|x_b|)+n^{O(\epsilon)}|x_a||x_b|\right)(s_2-s_1)^{1/2}. \label{t.t.t.t}
\end{eqnarray}
Note that (\ref{t.t.t.t}) is a result of fact that $p=q=1$ for $[-RVR]_{\mathbf{x}\mathbf{x}}$ (See page 22 for the definitions of $p$ and $q$).
Thus when $\{k_1,k_2,k_3,k_4\}=\{0,1,1,3\}$, one has
\begin{eqnarray*}
&&n^{-5/2}\sum_{a,b}|\mathbb{E}
\Im\mathcal{F}_{0}(\Im\mathcal{F}_{1})^2
\Im\mathcal{F}_{3}|\nonumber\\
&\leq & n^{-5/2}\sum_{a,b}n^{-3/8+O(\epsilon)}\left(n^{-1/2+O(\epsilon)}+n^{-1/4+O(\epsilon)}(|x_a|+|x_b|)+n^{O(\epsilon)}|x_a||x_b|\right)^2\nonumber\\
&&\times \left(n^{-1/2+O(\epsilon)}+n^{O(\epsilon)}(|x_a|^2+|x_b|^2)\right)(s_2-s_1)^2\nonumber\\
&\leq & n^{-5/2}\sum_{a,b}n^{-3/8+O(\epsilon)}\left(n^{-1
+O(\epsilon)}+n^{-1/2+O(\epsilon)}(|x_a|^2+|x_b|^2)+n^{O(\epsilon)}|x_a|^2|x_b|^2\right)\nonumber\\
&&\times \left(n^{-1/2+O(\epsilon)}+n^{O(\epsilon)}(|x_a|^2+|x_b|^2)\right)(s_2-s_1)^2\nonumber\\
&\leq & n^{-19/8+O(\epsilon)}(s_2-s_1)^2.
\end{eqnarray*}
For $\{k_1,k_2,k_3,k_4\}=\{0,1,2,2\}$, one has
\begin{eqnarray*}
&&n^{-5/2}\sum_{a,b}|\mathbb{E}\Im\mathcal{F}_0\Im\mathcal{F}_1
(\Im\mathcal{F}_2)^2|\nonumber\\
&\leq & n^{-5/2}\sum_{a,b}n^{-3/8+O(\epsilon)}\left(n^{-1/2+O(\epsilon)}+n^{-1/4+O(\epsilon)}(|x_a|+|x_b|)+n^{O(\epsilon)}|x_a||x_b|\right)\nonumber\\
&&\times\left (n^{-1/2+O(\epsilon)}+n^{O(\epsilon)}(|x_a|^2+|x_b|^2)\right)^2(s_2-s_1)^2\nonumber\\
&\leq & n^{-5/2}\sum_{a,b}n^{-3/8+O(\epsilon)}\left(n^{-1/2+O(\epsilon)}+n^{-1/4+O(\epsilon)}(|x_a|+|x_b|)+n^{O(\epsilon)}|x_a||x_b|\right)\nonumber\\
&&\times \left(n^{-1+O(\epsilon)}+n^{O(\epsilon)}(|x_a|^4+|x_b|^4)\right)(s_2-s_1)^2\nonumber\\
&\leq& \bigg{(}n^{-5/2}\sum_{a,b}n^{-3/8+O(\epsilon)}|x_a||x_b|(n^{-1+O(\epsilon)}+|x_a|^4+|x_b|^4)+n^{-17/8+O(\epsilon)}\bigg{)}(s_2-s_1)^2.
\end{eqnarray*}
Note that $\sum_{a}|x_a|=O(\sqrt{n})$. Thus
\begin{eqnarray*}
\sum_{a,b}|x_a||x_b|= O(n),\quad \sum_{a,b}|x_a|^5|x_b|= O(\sqrt{n}).
\end{eqnarray*}
Thus we have
\begin{eqnarray*}
n^{-5/2}\sum_{a,b}|\mathbb{E}\Im\mathcal{F}_0\Im\mathcal{F}_1
(\Im\mathcal{F}_2)^2|\leq n^{-17/8+O(\epsilon)}(s_2-s_1)^2.
\end{eqnarray*}
The estimation towards case ({\bf{iv}}) is similar. We do it as follows.
\begin{eqnarray*}
&&n^{-5/2}\sum_{a,b}|\mathbb{E}(\Im\mathcal{F}_1)^3\Im\mathcal{F}_2|\nonumber\\
&\leq & n^{-5/2}\sum_{a,b}\left(n^{-1/2+O(\epsilon)}+n^{-1/4+O(\epsilon)}(|x_a|+|x_b|)+n^{O(\epsilon)}|x_a||x_b|\right)^3\nonumber\\
&&\times \left(n^{-1/2+O(\epsilon)}+n^{O(\epsilon)}(|x_a|^2+|x_b|^2)\right)(s_2-s_1)^2\nonumber\\
&\leq&n^{-5/2}\sum_{a,b}
\left(n^{-3/2+O(\epsilon)}+n^{-3/4+O(\epsilon)}(|x_a|^3+|x_b|^3)+n^{O(\epsilon)}|x_a|^3|x_b|^3\right)\nonumber\\
&&\times \left(n^{-1/2+O(\epsilon)}+n^{O(\epsilon)}(|x_a|^2+|x_b|^2)\right)(s_2-s_1)^2\nonumber\\
&\leq &\bigg{(}n^{-5/2}\sum_{a,b}|x_a|^3|x_b|^3(n^{-1/2+O(\epsilon)}+n^{O(\epsilon)}(|x_a|^2+|x_b|^2))+n^{-9/4+O(\epsilon)}\bigg{)}(s_2-s_1)^2\nonumber\\
&\leq &n^{-9/4+O(\epsilon)}(s_2-s_1)^2.
\end{eqnarray*}
In summary, we have
\begin{eqnarray}
|\mathbb{E}n^{-5/2}\mathbf{1}_
{\{\sum_{i=1}^4k_i=5\}}
\prod_{i=1}^4\Im\mathcal{F}_{k_i}|
=n^{-17/8+O(\epsilon)}(s_2-s_1)^2. \label{s.9}
\end{eqnarray}

Now we come to deal with the case of $\kappa=6$. When there is at least one $k_i$ equal to $0$, it will be easy to check that the contributions of such terms are negligible by using (\ref{t.6}) and (\ref{p.255}). We leave the details to the readers. Now we still have the case where there is no $0$ in $\{k_1,k_2,k_3,k_4\}$. Since $\kappa=6$, we have
\begin{eqnarray*}
\{k_1,k_2,k_3,k_4\}=
\{1,1,2,2\}~or~\{1,1,1,3\}.
\end{eqnarray*}
The discussions for these two cases are quite similar to that of ({\bf{iv}}) for $\kappa=5$. Actually, we can get
\begin{eqnarray}
|\mathbb{E}n^{-3}\mathbf{1}_
{\{\sum_{i=1}^4k_i=6\}}
\prod_{i=1}^4\Im\mathcal{F}_{k_i}|=n^{-5/2+O(\epsilon)}(s_2-s_1)^2. \label{u.9}
\end{eqnarray}
Then by (\ref{t.9}), (\ref{s.9}) and (\ref{u.9}) we can complete the comparison procedure. Thus (\ref{10.4}) follows. So does Lemma \ref{lem.t.0}.
\end{proof}


\begin{thebibliography}{00}
\bibitem{BP2012}
Z.D. Bai, G.M. Pan, Limiting behavior of eigenvectors of large wigner matrices. J. Stat. Phys. 146: 519-549. (2012)
\bibitem{BG2012}
F. Benaych-Georges. A universality result for the global fluctuations of the eigenvectors of Wigner matrices. Preprint
\bibitem{BP20121}
F. Benaych-Georges, S. P\'{e}ch\'{e}. Localization and delocalization for heavy tailed band matrices. arXiv:1210.7677.
\bibitem{Billingsley1968}
P. Billingsley. Convergence of probability measure. Wiley, New York. (1968)
\bibitem{BG20121}
C. Bordenave, A. Guionnet. Localization and delocalization of eigenvectors for heavy-tailed random matrices. arXiv:1201.1862, 2012
\bibitem{EK2011}
L. Erd\"{o}s, A. Knowles. Quantum diffusion and eigenfunction delocalization in a random band matrix model. Commun. Math. Phys. 303, 2, 509-554. (2011)
\bibitem{ESY2009}
L. Erd\"{o}s, B. Schlein, and H-T. Yau. Semicircle law on short scales and delocalization of eigenvectors for Wigner random matrices. Ann. Probab. 37, no. 3, 815-852 (2009)
\bibitem{ESY20092}
L. Erd\"{o}s, B. Schlein, and H-T. Yau. Local semicircle law and complete delocalization for Wigner random matrices. Commun. Math. Phys. 287, 641-655 (2009).
\bibitem{EYY2010}
L. Erd\"{o}s, H-T. Yau, and J. Yin. Bulk universality for generalized Wigner matrices. Probab. Theory Relat. Fields, DOI 10.1007/s00440-011-0390-3.
\bibitem{EYY2012}
L. Erd\"{o}s, H-T. Yau, and J. Yin. Rigidity of eigenvalues of generalized Wigner matrices. Advances in Mathematics, 229, 1435-1515, (2012).
\bibitem{KY2011}
A. Knowles, J. Yin. Eigenvector distribution of Wigner matrices. Probab. Theory Relat. Fields, DOI 10.1007/s00440-011-0407-y.
\bibitem{KY2012}
A. Knowles, J. Yin. The isotropic semicircle law and deformation of Wigner matrices. Preprint
\bibitem{S2009}
J. Schenker. Eigenvector localization for random band matrices with power law band width. Commun. Math. Phys.
290, Issue 3, 1065-1097. (2009)
\bibitem{Silverstein1990}
J.W. Silverstein. Weak convergence of random functions defined by the eigenvectors of sample covariance matrices.
Ann. Probab. 18, No. 3, 1174-1194.  (1990)
\bibitem{TV2011}
T. Tao, V. Vu. Random matrices: Universality of local eigenvalue statistics. Acta. Math., 206, 127-204, (2011)
\bibitem{TV2010}
T. Tao, V. Vu. Random matrices: Universality
of local eigenvalue statistics up to the edge. Comm. Math. Phys. vol 298, No.2, 549-572 (2010)
\bibitem{TV2012}
T. Tao, V. Vu. Random matrices: Universal properties of eigenvectors. Random Matrices: Theory Appl. 01, 1150001 (2012)
\end{thebibliography}
\end{document}